\documentclass[10pt, a4paper, twoside]{amsart}
\usepackage{amsfonts}
\usepackage{mathrsfs}
\usepackage{amsmath}
\usepackage{amssymb}
\usepackage{fancyhdr}
\usepackage{graphicx}
\usepackage{color}

\numberwithin{equation}{section}

\allowdisplaybreaks

\newcommand\RR{\mathbb{R}}
\newcommand\CC{\mathbb{C}}
\newcommand\NN{\mathbb{N}}
\newcommand\ZZ{\mathbb{Z}}
\newcommand\HH{\mathbb{H}}
\newcommand\Id{\mathrm{Id}}

\newcommand\ang[1]{\langle #1 \rangle}

\newcommand\Omegazh{{}^0\Omega^{1/2}}

\newcommand\diag{\mathrm{diag}}
\newcommand\diagz{\mathrm{diag}_0}

\newcommand\FL{\mathrm{FL}}
\newcommand\FR{\mathrm{FR}}
\newcommand\FF{\mathrm{FF}}
\newcommand\tL{\tilde \Lambda}

\newcommand\FBR{\mathrm{FBR}}

\renewcommand\Re{\operatorname{Re}}
\renewcommand\Im{\operatorname{Im}}
\newcommand\BR{\operatorname{BR}}
\newcommand\WF{\text{WF}_h}
\newcommand\supp{\operatorname{supp}}
\newcommand\blambda{{\boldsymbol\lambda}}

\newcommand\bsigma{{\boldsymbol\sigma
}}
\newcommand\bzeta{{\boldsymbol\zeta
}}
\newcommand\sqrtdelta{P}
\newcommand\spmeas{dE_{\sqrtdelta}(\boldsymbol\lambda)}

\newtheorem{theorem}{Theorem}
\newtheorem{lemma}[theorem]{Lemma}
\newtheorem{proposition}[theorem]{Proposition}
\newtheorem{corollary}[theorem]{Corollary}

\theoremstyle{remark}
\newtheorem{remark}[theorem]{Remark}

\begin{document}

\title{\textbf{Resolvent and Spectral Measure on Non-Trapping Asymptotically Hyperbolic Manifolds II: Spectral Measure, Restriction Theorem, Spectral Multipliers} }

\author{Xi Chen and Andrew Hassell}

\date{}

\begin{abstract} We consider  the Laplacian $\Delta$ on an asymptotically hyperbolic manifold $X$, as defined by Mazzeo and Melrose \cite{Mazzeo-Melrose}.
We give pointwise bounds on the Schwartz kernel of the spectral measure for the operator $(\Delta - n^2/4)_+^{1/2}$ on such manifolds, under the assumptions that $X$ is nontrapping and there is no resonance at the bottom of the spectrum. This uses the construction of the resolvent given by Mazzeo and Melrose \cite{Mazzeo-Melrose} (valid when the spectral parameter lies in a compact set),  Melrose, S\'{a} Barreto and Vasy \cite{Melrose-Sa Barreto-Vasy} (high energy estimates for a perturbation of the hyperbolic metric) and the present authors \cite{Chen-Hassell1} (see also \cite{Wang}) in the general high-energy case.

{We give two applications of the spectral measure estimates. The first,  following work due to Guillarmou and Sikora with the second author \cite{Guillarmou-Hassell-Sikora} in the asymptotically conic case, is a restriction theorem, that is, a $L^p(X) \to L^{p'}(X)$ operator norm bound on the spectral measure. The second is a  spectral multiplier result under the additional assumption that $X$ has negative curvature everywhere, that is, a bound on functions $F((\Delta - n^2/4)_+^{1/2})$ of the square root of the Laplacian, in terms of norms of the function $F$.}
Compared to the asymptotically conic case, our spectral multiplier result is weaker, but the restriction estimate is stronger. In both cases, the difference can be traced to the exponential volume growth at infinity for asymptotically hyperbolic manifolds, as opposed to polynomial growth in the asymptotically conic setting.

The pointwise bounds on the spectral measure established here will also be applied to Strichartz estimates in \cite{Chen-Hassell3}.\end{abstract}

\maketitle

\section{Introduction}
This paper, following \cite{Chen-Hassell1}, is the second in a series of three devoted to the analysis of the resolvent family and spectral measure for the Laplacian on an asymptotically hyperbolic, nontrapping manifold. The third paper, by the first author alone, will establish global-in-time Strichartz estimates on such a manifold.

Let $(X^\circ, g)$ be an asymptotically hyperbolic manifold of
dimension $n+1$ (see Section~\ref{subsec:ahm} for the precise
definition of `asymptotically hyperbolic'). Let $\Delta$ be the
positive Laplacian on $(X^\circ, g)$, which is essentially self-adjoint on
$C_c^\infty(X^\circ)$. It is well known that the spectrum of
$\Delta$ is absolutely continuous on $[n^2/4, \infty)$ \cite{Mazzeo-1991} with possibly
finitely many eigenvalues (of finite multiplicity) in $(0, n^2/4)$.
We write $\sqrtdelta$ for the operator 
\begin{equation}
\sqrtdelta = 
(\Delta - n^2/4)_+^{1/2},
\label{P}\end{equation}
where the subscript $+$ indicates positive part --- thus, $\sqrtdelta$ vanishes on the pure point eigenspaces. In this paper, we
analyze the spectral measure $\spmeas$ of the operator $\sqrtdelta$,
under the assumption that $(X^\circ, g)$ is nontrapping (that is,
every geodesic reaches infinity both forward and backward) and that
there is no resonance at the bottom of the continuous spectrum,
$n^2/4$.  To do this, we express the spectral measure $\spmeas$ in
terms of the boundary values of the resolvent $(\Delta - n^2/4 -
(\blambda \pm i0)^2)^{-1}$ just `above' and `below' the spectrum in
$\CC$. We then use the construction of the resolvent given by Mazzeo
and Melrose \cite{Mazzeo-Melrose} (valid when the spectral parameter
lies in a compact set),  Melrose, S\'{a} Barreto and Vasy
\cite{Melrose-Sa Barreto-Vasy} (high energy estimates for a
perturbation of the hyperbolic metric) and the present authors
\cite{Chen-Hassell1} (and, independently,  \cite{Wang}) in the general
high-energy case to get precise information about the Schwartz
kernel of the spectral measure. In particular, following the work of
the second author with Guillarmou and Sikora
\cite{Guillarmou-Hassell-Sikora} in the asymptotically conic
setting, this will allow us to obtain precise pointwise bounds on
the Schwartz kernel, when (micro)localized near the diagonal in a
certain sense.

We then apply these pointwise kernel bounds to prove operator norm estimates on the spectral measure $\spmeas$, and on general functions $F(\sqrtdelta)$ of the operator $\sqrtdelta$, again following the general strategy of \cite{Guillarmou-Hassell-Sikora}.
However, there are key differences in the results we prove here compared to the asymptotically conic case, which can be traced to the exponential, as opposed to polynomial, growth of the volume of large balls in the present setting.
In the case of the restriction theorem, that is, an $L^p \to L^{p'}$ bound on the spectral measure, we prove more: we obtain an estimate for all $p \in [1, 2)$, while in the asymptotically conic case, it is well known that such an estimate fails for $p \geq 2(d+1)/(d+3)$, where $d$ is the dimension. In the case of the spectral multiplier result, that is, boundedness of $F(\sqrtdelta)$, where we assume only a finite amount of Sobolev regularity on $F$, boundedness on $L^p(X)$ spaces fails for $p \neq 2$ due to results of Clerc-Stein  \cite{C-S} and Taylor \cite{T1}. Instead, we obtain boundedness on $L^p(X) + L^2(X)$ for $p \in [1, 2)$, {provided $X$ is negatively curved}.

\subsection{The spectral measure}

Consider functions of an abstract (unbounded) self-adjoint operator
$L$ on a Hilbert space $H$. These are defined by the spectral
theorem for unbounded self-adjoint operators (for example, see
\cite[p.263]{R-S}). One standard version of this theorem  says that
there is a one-to-one correspondence between self-adjoint operators
$L$ and increasing, right-continuous families of projections
$E(\blambda)$, $\blambda \in \RR$, having the property that the
strong limit of $E(\blambda)$ as $\blambda \to -\infty$ is the zero
operator and as $\blambda \to +\infty$ is the identity. The
correspondence is given by
$$L = \int_{-\infty}^\infty \blambda \, d E(\blambda);$$ if $g(\cdot)$ is a real-valued Borel function on $\mathbb{R}$, then $$g(L) = \int_{-\infty}^\infty g(\blambda) \, dE(\blambda)  $$ with domain
$$\{\psi : \int_{-\infty}^\infty |g(\blambda)|^2 \, d \langle\psi, E(\blambda) \psi\rangle < \infty\}$$
is self-adjoint. Here the formula means $$\langle g(L)\psi, \psi
\rangle = \int_{-\infty}^\infty g(\blambda) \, d \langle E(\blambda)
\psi, \psi\rangle,$$
which can be interpreted as a Stieltjes integral since $\langle E(\blambda)
\psi, \psi\rangle$ is a nondecreasing function of $\blambda$. 
 We call $dE(\blambda)$ the spectral measure
associated with the operator $L$.

In particular we can apply this when $L = P$ and $H = L^2(X, g)$. We
then write $\spmeas$ for the spectral measure of $P$. Since $P$ is a
positive operator, we only need to integrate over $\blambda \in [0,
\infty)$ in this case.

Returning to the abstract operator $L$, the resolvent family $(L -
\blambda)^{-1}$ is a holomorphic family of bounded operators on $H$
for $\Im \blambda \neq 0$. In many cases, including in the present
setting, the resolvent family extends continuously to the real axis
as a bounded operator in a weaker sense, e.g. between weighted $L^2$
spaces, and is then differentiable in $\blambda$ up to the real
axis. In that case, we find that $E(\blambda)$ is differentiable in
$\blambda$ and we have  Stone's formula
\begin{equation}
\frac{d}{d\blambda} E(\blambda) = \frac1{2\pi i} \Big( (L -
(\blambda + i0))^{-1} - (L - (\blambda - i0))^{-1} \Big).
\label{Stone}\end{equation}

In this case we write (abusing notation somewhat) $dE(\blambda)$ for
the derivative of $E(\blambda)$ with respect to $\blambda$. Stone's
formula gives a mechanism for analyzing the spectral measure, namely
we need to analyze the limit of the resolvent $(L - \blambda)^{-1}$ on the real axis. In
the case of $P$, we notice that the spectral measure $\spmeas$ for
$P$ is $2\blambda$ times the spectral measure at $n^2/4 +
\blambda^2$ for $\Delta$. This gives us the formula
\begin{equation}
 \spmeas = \frac{\blambda}{\pi i} \Big( (\Delta - (n^2/4 + \blambda^2 + i0))^{-1} - (\Delta - (n^2/4 + \blambda^2 - i0))^{-1} \Big).
\label{Stone2}\end{equation}

\subsection{Restriction theorem via spectral measure}

Stein \cite{Beijing lecture} and Tomas \cite{Tomas} proved estimates for the restriction of the Fourier transform of an $L^p$ function to the sphere $\mathbb{S}^{d-1} \subset \RR^d$: 
$$\int_{\mathbb{S}^{d - 1}} |\hat{f}|^2 \, d\sigma \leq C \|f\|_{L^p(\mathbb{R}^d)}^2, \quad
p \in [1, 2(d + 1)/(d + 3)].
$$ 
 Alternatively, we may formulate the estimate in terms of the the restriction operator $R$ to the hypersphere,  $$R(f)(\xi) = \int_{\mathbb{R}^d} e^{- i x \cdot \xi} f(x) \, dx, \quad |\xi| = 1.$$ 
 The Stein-Tomas theorem is equivalent to the boundedness of $R : L^p(\mathbb{R}^d) \longrightarrow L^{2}(\mathbb{S}^{d-1})$, which in turn is equivalent to the boundedness of 
 $R^\ast R : L^p(\mathbb{R}^d) \longrightarrow L^{p^\prime}(\mathbb{R}^d).$ The Schwartz kernel of $R^\ast R$, $$ \int_{|\xi| = 1} e^{i (x - y) \cdot \xi} \, d\xi,$$ is $(2\pi)^{d}$ times the spectral measure $dE_{\sqrt{\Delta}}(1)$ for the square root of the flat Laplacian on $\RR^d$, since the spectral projection $E_{\sqrt{\Delta}}(\blambda)$ of $\sqrt{\Delta}$ is $\mathcal{F}^{-1}(\chi_{B(0, \blambda)}) \mathcal{F}$. Therefore, one may rewrite the restriction theorem as the following estimate:  \begin{equation}\label{Euclidean restriction}  \|dE_{\sqrt{\Delta}}(\blambda)\|_{L^p \rightarrow L^{p^\prime}} \ ( \ = \blambda^{d(2/p - 1) - 1} \|dE_{\sqrt{\Delta}}(1)\|_{L^p \rightarrow L^{p^\prime}} \ )  \ \leq C \blambda^{d(2/p - 1) - 1},\end{equation} provided $p \in [1, 2(d + 1)/(d + 3)].$
 This naturally leads to the question: for which Riemannian manifolds $(N, g)$ does the spectral measure for $\sqrt{\Delta_{N,g}}$ map $L^p(N,g)$ to $L^{p'}(N,g)$ for some $p \in [1, 2)$, and how does the norm depend in the spectral parameter?  We refer to such an estimate as a `restriction estimate' or a `restriction theorem'. Such a result is a continuous spectral analogue of the well-known discrete restriction theorem of Sogge \cite[Chapter 5]{sogge}.

\subsection{Results on asymptotically conic spaces}
As the present paper is inspired by work by the second author with Guillarmou and Sikora \cite{Guillarmou-Hassell-Sikora} on asymptotically conic spaces, we review the results of \cite{Guillarmou-Hassell-Sikora} here.

Asymptotically conic spaces $M$, of dimension $m$, are modelled on spaces that at infinity look like the `large end of a cone'; that is, have one end diffeomorphic to $(r_0, \infty) \times Y$, where $Y$ is a closed manifold of dimension $m - 1$, with
 a metric of the form
$$
dr^2 + r^2 g_0(y, dy) + O(\frac1{r}), r \to \infty,
$$
where $g_0$ a metric on $Y$. Such spaces are Euclidean-like at infinity, in the sense that the volume of balls of radius $\rho$ are uniformly bounded above and below by multiples of $\rho^m$, and in the sense that the curvature tends to zero, and the local injectivity radius tends to infinity, at infinity. If we add the condition that the manifold be nontrapping, then such spaces are also dynamically similar to Euclidean space (although they may have conjugate points).
Consequently, the spectral analysis of such spaces behaves in many ways like Euclidean space. This is illustrated by the results from \cite{Guillarmou-Hassell-Sikora}. On $\RR^m$, the spectral measure satisfies pointwise kernel bounds of the form
\begin{eqnarray}
\bigg|\bigg(\frac{d}{d \blambda}\bigg)^{j}
dE_{\sqrt{\Delta}}(\blambda)(x, y) \bigg| &\leq& C \blambda^{m-1 -
j}(1 + \blambda |x-y|)^{- (m-1)/2 + j}, \quad j = 0, 1, 2, \dots ,
\label{spmeasRd}\end{eqnarray} and this estimate is essentially
optimal, in the sense that neither exponent can be improved. In
\cite{Guillarmou-Hassell-Sikora} it was shown that, if $M$ is an
asymptotically conic nontrapping manifold, and $\Delta$ its
Laplacian, then there is a partition of unity $\Id = \sum_{j=0}^N
Q_i(\blambda)$, depending on $\blambda$, and $\delta > 0$ such that
\begin{equation} \bigg| Q_i(\bsigma) \bigg( \Big( \frac{d}{d \blambda} \Big)^{j}dE_{\sqrt{\Delta}}(\blambda) \bigg) Q_i^*(\bsigma) (x, y)\bigg| \leq C \blambda^{m - 1 - j}(1 + \blambda d(x, y))^{- (m-1)/2 + j}, \quad j = 0, 1, 2, \dots , \label{as-conic-spe}\end{equation}
for $\bsigma \in [(1-\delta) \blambda, (1+\delta) \blambda]$, where $d(x, y)$ is the Riemannian distance\footnote{This was only claimed for $\blambda = \bsigma$ in \cite{Guillarmou-Hassell-Sikora}, but in \cite{Guillarmou-Hassell-2014} it was observed that the same construction gives the more general estimates in \eqref{as-conic-spe}.}. The $Q_i(\blambda)$ are
semiclassical pseudodifferential operators (with semiclassical
parameter $ h = \blambda^{-1})$ with small microsupport. Therefore,
the operators $Q_i(\bsigma) dE_{\sqrt{\Delta}}(\blambda) Q_i^*(\bsigma)$ can be
considered to be the kernel of the spectral measure (micro)localized
near the diagonal. Moreover, in the case where there are no
conjugate points, then the estimate above is valid without the
partition of unity.

This estimate \eqref{as-conic-spe} was shown to imply a global
restriction estimate, that is, an $L^p(M) \to L^{p'}(M)$ operator
norm bound on $dE_{\sqrt{\Delta}}(\blambda)$. In fact, this was
proved at an abstract level:

\begin{theorem}[\cite{Guillarmou-Hassell-Sikora, restriction notes}]\footnote{This theorem was formulated and partially proved in \cite{Guillarmou-Hassell-Sikora}. See \cite{restriction notes} for a complete proof. }\label{kernel2restriction}
Let $(X, d, \mu)$ be a metric measure space, and $L$ an abstract
positive self-adjoint operator on $L^2(X, \mu)$. Suppose that the
spectral measure $dE_{\sqrt{L}}(\blambda)$ has a Schwartz kernel
satisfying \eqref{spmeasRd} (with $|x-y|$ replaced by $d(x, y)$) for
$j=0$, as well as for $j=d/2-1$ and $j=m/2$ if $m$ is even, or
$j=m/2-3/2$ and $j=m/2+1/2$ if $m$ is odd. Then the operator norm
estimate
\begin{equation}
\big\| dE_{\sqrt{\Delta}}(\blambda) \big\|_{L^p(M) \to L^{p'}(M)}
\leq C \blambda^{m(1/p - 1/p') - 1}, \quad 1 \leq p \leq
\frac{2(m+1)}{m+3}, \label{spmeasX}\end{equation} holds for all
$\blambda > 0$. Moreover, if the kernel estimates above hold for
some range of $\blambda$, then \eqref{spmeasX} holds for $\blambda$
in the same range.

\end{theorem}

Finally, it was shown in \cite{Guillarmou-Hassell-Sikora} that, at an abstract level, such a restriction estimate implies spectral multiplier estimates:

\begin{theorem}[\cite{Guillarmou-Hassell-Sikora}]\label{restriction2multiplier} Let $(X, d, \mu)$ be a metric measure space,  such that the volume of each ball of radius $\rho$ is comparable to $\rho^m$.
Suppose $\Delta$ is a positive self-adjoint operator such that $\cos
t \sqrt{\Delta}$ satisfies finite propagation speed on $L^2(X)$, and
the restriction theorem $$\|dE_{\sqrt{\Delta}}(\blambda)\|_{L^p
\rightarrow L^{p^\prime}} \leq C \blambda^{m(1/p - 1/p^\prime) -
1}$$ holds uniformly with respect to $\blambda > 0$ for $1 \leq p
\leq 2(m+1)/(m+3)$. Then  there is a uniform operator norm bound on
spectral multipliers on $L^p(X)$ of the form
\begin{equation}
\sup_{\alpha > 0} \|F(\alpha \sqrt{\Delta})\|_{L^p \rightarrow L^p} \leq C \|F\|_{H^s},
\label{sm}\end{equation} where $F \in H^s(\mathbb{R})$ is an even function supported in $[-1, 1]$, and $s > m (1/p - 1/2 )$.\end{theorem}

In particular, one concludes \eqref{spmeasX} and \eqref{sm} when $X$ is an asymptotically conic nontrapping manifold of dimension $d$.

\subsection{Hyperbolic space}
We next consider existing results on hyperbolic space.
We return to our convention where the dimension is $n+1$.
Using explicit formulae for the Schwartz kernel of functions of the operator $P = (\Delta - n^2/4)^{1/2}$, we deduce pointwise bounds
\begin{equation}
\bigg|dE_{P}(\blambda)(z,z') \bigg| \leq \begin{cases} C \blambda^{2}, \quad d(z,z') \leq 1 \\
C \blambda^{2} d(z,z') (1 +\blambda d(z,z'))^{-1} 
e^{-nd(z,z')/2}, \ d(z,z') \geq 1 \end{cases} 
\end{equation}
for $\blambda \leq 1$, and derivative estimates\footnote{We can obtain derivative estimates for  $\blambda \leq 1$
also, but we do not need such estimates in the low energy case.}
\begin{eqnarray}
\big|\big(\frac{d}{d \blambda}\big)^{j} dE_{P}(\blambda)(z,z') \big|
&\leq& 
\begin{cases}
C \blambda^{n-j} (1 + d(z,z') \blambda)^{-n/2+j}, \text{ for } d(z,z') \leq 1 \\
C \blambda^{n/2} d(z,z')^{j} e^{-nd(z,z')/2}, \text{ for } d(z,z')
\geq 1.
\end{cases}
\end{eqnarray}
when $\blambda \geq 1$. Closely related pointwise bounds for the
wave kernels $\cos tP$ and $P^{-1}\sin tP$, the heat kernel
$e^{-tP^2}$ and the Schr\"odinger propagator $e^{itP^2}$ on
hyperbolic space have been exploited in various works; see for
example \cite{Bunke-Olbrich}, \cite{Davies-Mandouvalos},
\cite{Anker-Pierfelice}, \cite{Burq-Guillarmou-Hassell}. 

To the authors' knowledge, the recent paper \cite{Huang-Sogge} by Huang and Sogge is the 
only previous paper in which restriction estimates for hyperbolic
space have been considered. Huang and Sogge proved restriction estimates
for $p$ in the same range $[1, 2(d+1)/(d+3)]$ as for Euclidean space,
using the exact expression for the hyperbolic resolvent, and complex interpolation,
in the manner of Stein's original proof of the Stein-Tomas
restriction theorem \cite{Tomas} (this argument was presented in an
abstract formulation in \cite{Guillarmou-Hassell-Sikora}). 
In fact, on hyperbolic space (and, as we shall show, asymptotically
hyperbolic nontrapping spaces), restriction estimates are valid
for all $p \in [1,2)$ --- see Section~\ref{sec:model} for a very simple proof on $\HH^3$.


Spectral multiplier estimates on hyperbolic and asymptotically hyperbolic spaces on $L^p$ spaces (much more general than those considered here) have been well studied. It was pointed out by Clerc and Stein \cite{C-S} for symmetric spaces and Taylor \cite{T1} for spaces with exponential volume growth and $C^\infty$ bounded geometry that a \emph{necessary} condition for $F(P)$ to be bounded is that $F$ admit an analytic continuation to a strip in the complex plane. Cheeger, Gromov and Taylor \cite{C-G-T}, and Taylor \cite{T1} showed
that  if $M$ has $C^\infty$ bounded geometry and injectivity radius bounded from below, then $F(\sqrt{\Delta})$ maps $L^p(M)$ into itself for $1 < p < \infty$, provided that $F$ is holomorphic and even on the strip $\{z\in\mathbb{C} : |\mathrm{Im} z| < W\}$ for some $W$ and satisfies symbol estimates
$$|F^{(j)}(z)| \leq C_j \langle z \rangle^{k - j}$$
on the strip.

By constrast, we want to consider the mapping properties of $F(P)$ where $F$ has only  finite Sobolev regularity. This is motivated by typical applications of spectral multipliers in harmonic analysis, such as Riesz means, and in PDE, in which one often wants to restrict to a dyadic frequency interval, that is, to the range of a spectral projector of the form $1_{[2^j, 2^{j+1}]}(P)$, or a smoothed version of this. Clearly, such a spectral multiplier cannot have an analytic continuation to a strip. On the other hand, the work of Clerc-Stein and Taylor shows that boundedness on $L^p$, $p \neq 2$, cannot be expected. This motivates us to search for replacements for $L^p$ spaces, on which spectral multipliers are bounded.

\subsection{Asymptotically hyperbolic manifolds}\label{subsec:ahm}
The geometric setting in the present paper is that of asymptotically hyperbolic manifolds. An asymptotically hyperbolic manifold $(X^\circ, g)$ is the interior of a compact manifold $X$ with boundary, such that the Riemannian metric $g$ takes a specific degenerate form near the boundary of $X$. Specifically, near each boundary point, there are local coordinates $(x, y)$, where $x$ is a boundary defining function and $y$ restrict to local coordinates on $\partial X$, such that $g$ takes the form
\begin{equation}
g = \frac{d x^2 + g_0(x, y, dy)}{x^2}.
\label{metric}\end{equation}
where $g_0(x, y, dy)$ is a family of metrics on $\partial X$, smoothly parametrized by $x$. Under the metric $g$, the interior $X^\circ$ of $X$ is a complete Riemannian manifold.

As is well known, $n+1$-dimensional hyperbolic space takes this form in the Poincar\'e ball model. Indeed, $\HH^{n+1}$ is given by the interior of the unit ball in $\RR^{n+1}$, with the metric
\begin{equation}
g = \frac{4 dz^2}{(1 - |z|^2)^2},
\label{Poincare}\end{equation}
where $z = (z_1, \dots, z_{n+1})$ are the standard coordinates on $\RR^{n+1}$. Other examples include all convex co-compact hyperbolic manifolds, and compactly supported metric perturbations of these.

Such spaces are termed 
asymptotically hyperbolic spaces as the sectional curvatures tend to $-1$ at infinity \cite{Mazzeo-Melrose}. 
Analytically, they have many similarities to hyperbolic spaces. Consider the resolvent
$R(\zeta) := (\Delta - \bzeta(n - \bzeta))^{-1}$ on $\HH^{n+1}$, which
is well-defined as a bounded operator on $L^2(\HH^{n+1})$ for $\Re
\bzeta > n/2$. Notice that the axis $\Re \bzeta = n/2$ corresponds to
the spectrum of $\Delta$, and the point $\bzeta = n/2 \pm i\blambda$
corresponds to the point $|\blambda|$ in the spectrum of $P =
(\Delta - n^2/4)_+^{1/2}$.
On $\HH^{n+1}$, the resolvent $R(\bzeta)$ extends to a holomorphic function of $\bzeta \in \CC$ when $n$ is even, and a meromorphic function with poles at $\{ 0, -1, -2, \dots \}$ when $n$ is odd.

 For asymptotically hyperbolic spaces, it is known from work of Mazzeo-Melrose \cite{Mazzeo-Melrose} and Guillarmou \cite{Guillarmou} that the resolvent $(\Delta - \bzeta(n - \bzeta))^{-1}$ extends to be a meromorphic function of $\bzeta$ on $\CC \setminus \{ (n-1)/2 - k \mid k = 1, 2, 3, \dots \}$, and extends to be meromorphic on the whole of $\CC$ provided that $g$ is even in $x$, that is, a smooth function of $x^2$. In addition, it is holomorphic in a neighbourhood of the spectral axis $\Re \bzeta = n/2$ except possibly at the point $n/2$ itself, corresponding to the bottom of the continuous spectrum, which could be a simple pole \cite{Bouclet}. In the present article, we shall \emph{assume} that the resolvent is holomorphic at $\bzeta = n/2$ as well. We point out that our estimates will certainly fail in the case of a zero-resonance, but weaker estimates will remain valid; see \cite{Jensen-Kato}, \cite{Guillarmou-Hassell} for an analysis of zero-resonances in the asymptotically Euclidean case.


\subsection{Main results}\label{subsec:mr}

\subsubsection{Pointwise estimates on the spectral measure}

Our first main result, analogous to \eqref{as-conic-spe},  is that there is a partition of the identity,
$\Id = \sum_{j=0}^N Q_i(\blambda)$ on $L^2(X)$ such that the
diagonal terms in the two-sided decomposition of $\spmeas$ satisfy the same
type of pointwise bounds as are valid on hyperbolic space. In fact,
following \cite{Hassell-Zhang}, we prove a slightly stronger result,
in which we retain information about the oscillatory nature of the
kernel as $\blambda \to \infty$.

Before stating the result, we refer to Section~\ref{sec:ahm} for the definition of the double space $X^2_0$, the blow-up of $X^2$ at the boundary of the diagonal; see Figure~\ref{fig}. This space has 3 boundary hypersurfaces: the lift to $X^2_0$ of the left and right boundaries in $X^2$, denoted $\FL$ and $\FR$, respectively, and the `front face' $\FF$ created by blowup. We denote boundary defining functions for these boundary hypersurfaces by $\rho_L$, $\rho_R$ and $\rho_F$ respectively. 

\begin{theorem}\label{thm:kernelbounds1}
Let $(X^\circ, g)$ be an asymptotically hyperbolic nontrapping
manifold with no zero-resonance, and let $\sqrtdelta$ be given by \eqref{P}. Then for low energies, $\blambda
\leq 1$, the Schwartz kernel of the spectral measure $dE_{\sqrtdelta}(\blambda)$ takes the form
\begin{equation}
dE_{P}(\blambda) (z,z') = \blambda \Big( (\rho_L
\rho_R)^{n/2+i\blambda} a(\blambda, z, z') - (\rho_L
\rho_R)^{n/2-i\blambda} a(-\blambda, z, z') \Big),
\end{equation}
where $a \in C^\infty([-1, 1]_{\blambda} \times X^2_0)$.

For high energies, $\blambda \geq 1$,
 one can choose a finite pseudodifferential operator partition of the identity operator,  $$Id = \sum_{k = 0}^{N} Q_k(\blambda),$$ such that the $Q_j$ are bounded on $L^p$, uniformly in $\lambda$, for each  $p \in (1, \infty)$, and such that the   microlocalized spectral measure, that is, any of the compositions $Q_k(\blambda) dE_{P}(\blambda) Q_k^\ast (\blambda)$, $0 \leq k \leq N$, takes the form
\begin{gather}
Q_k(\blambda) dE_{P}(\blambda) Q_k^\ast (\blambda)(z,z') = \blambda^n \Big(  \sum_{\pm} e^{\pm i \blambda d(z, z')} b_{\pm}(\blambda, z, z')  \Big) \label{bpm} \\
+  \,  (\rho_L \rho_R)^{n/2+i\blambda} \, a_+ + (\rho_L
\rho_R)^{n/2-i\blambda} \, a_- + (xx')^{n/2+i\blambda} \, \tilde a_+
+ (xx')^{n/2-i\blambda} \, \tilde  a_-  \label{apm}
\end{gather}
where $a_\pm$ is in $\blambda^{-\infty} C^\infty([0,
1]_{\blambda^{-1}} \times X^2_0)$ and $\tilde a_\pm$ is in
$\blambda^{-\infty} C^\infty([0, 1]_{\blambda^{-1}} \times X^2)$,
and
 the functions $b_{\pm}$ satisfy the following.
For small distance, $d(z,z') \leq 1$, we have
\begin{eqnarray} \Big| \frac{d^j}{d\blambda^j} b_{\pm}(\blambda, z, z')  \Big| \leq  C\blambda^{- j} \big(1 + \blambda d(z, z^\prime)\big)^{- n/2 }.
\label{bpm2}\end{eqnarray}

For $d(z,z') \geq 1$,  $b_{\pm}$ is $\blambda^{-n/2}$ times a smooth
function of $\blambda^{-1}$, decaying to order $n/2$ at $\FL$ and
$\FR$:
\begin{equation}
b_{\pm}(\blambda, z, z') \in \blambda^{-n/2} (\rho_L \rho_R)^{n/2}
C^\infty([0,1]_{\blambda^{-1}} \times X^2_0).
\label{bpm2.5}\end{equation} Moreover, if $(X^\circ, g)$ is in
addition simply connected with nonpositive sectional curvatures,
then the estimates above are true for the spectral meaure without
microlocalization, i.e. in this case we can take $\{ Q_i(\blambda)
\} $ to be the trivial partition of unity.
\end{theorem}

\begin{remark} We can split the continuous spectrum of $P$ at any point $\blambda \in (0, \infty)$ to differentiate high and low energies.\end{remark}

Using this structure theorem, we prove pointwise bounds on the microlocalized spectral measure:

\begin{theorem}\label{thm:kernelbounds}
Let $(X^\circ, g)$ be as above. Then for low energies, $\blambda
\leq 1$, we have pointwise estimates on the spectral measure of the
form
\begin{equation}
\Big|dE_{P}(\blambda) (z,z') \Big| \leq \begin{cases} C \blambda^{2} , \quad  d(z,z') \leq 1 \\
C \blambda^{2} d(z,z') (1+ \blambda d(z,z'))^{-1} e^{-nd(z,z')/2}, \ d(z,z') \geq 1. \end{cases}
\end{equation}

For high energies, $\blambda \geq 1$, one has, for sufficiently small $\delta > 0$ and $\bsigma \in [(1-\delta)\blambda, (1+\delta)\blambda]$
\begin{equation}\begin{gathered}
\Big|  Q_k(\bsigma)
\bigg( \big( \frac{d}{d\blambda} \big)^j dE_{P}(\blambda) \bigg) Q_k^\ast (\bsigma)(z,z')  \Big| \leq
\begin{cases}
C \blambda^{n-j} (1 + d(z,z') \blambda)^{-n/2+j}, \text{ for } d(z,z') \leq 1 \\
C \blambda^{n/2} d(z,z')^{j} e^{-nd(z,z')/2}, \text{ for } d(z,z')
\geq 1.
\end{cases}
\end{gathered}\label{bpm4}\end{equation}
As before,  if $(X^\circ, g)$ is in addition simply connected with
nonpositive sectional curvatures, then the estimates above are true
for the spectral meaure without microlocalization, i.e. in this case
we can take $\{ Q_i(\blambda) \} $ to be the trivial partition of
unity.
\end{theorem}

\subsubsection{Restriction theorem}

Using Theorem~\ref{thm:kernelbounds}, we prove

\begin{theorem}\label{thm:restriction} Suppose $(X, g)$ is an $n + 1$-dimensional non-trapping asymptotically hyperbolic manifold with no resonance at the bottom of the continuous spectrum. Then we have the following estimate for $\blambda \leq 1$:
\begin{equation}
\|dE_{P} (\blambda)\|_{L^p \rightarrow L^{p^\prime}} \leq C
\blambda^2, \quad 1 \leq p < 2.
\end{equation}

For $\blambda \geq 1$, we have the estimate
\begin{equation}\label{eqn:high energy restriction}
\|dE_{P} (\blambda)\|_{L^p \rightarrow L^{p^\prime}} \leq
\begin{cases}
C \blambda^{(n+1)(1/p - 1/p') - 1}, \quad 1 \leq p \leq \frac{2(n+2)}{n+4}, \\
C \blambda^{n(1/p - 1/2)}, \quad \frac{2(n+2)}{n+4} \leq p < 2.
\end{cases}
\end{equation}

\end{theorem}

\begin{remark} The range of exponents $p$ is greater for a hyperbolic space than for a conic (Euclidean) space. Indeed,  it includes all $p < 2$, while on Euclidean space $\RR^d$, the well-known Knapp example shows that the restriction estimate cannot hold for $p > 2(d+1)/(d+3)$. (The Knapp example does not apply to hyperbolic space as it relies on the dilation symmetry of $\RR^d$.) For high energies, $\blambda \geq 1$, the exponent is the same as on $\RR^d$ for the range $1 \leq p \leq 2(d+1)/(d+3)$ but again we get the full range of $p$ up to $p=2$.

This surprising result is closely tied to a non-Euclidean feature of hyperbolic space related to the Kunze-Stein phenomenon \cite{Kunze-Stein}. The Kunze-Stein phenomenon for semisimple Lie groups is that there is a much larger set of exponents $p, q, r$ for which one has
$$
L^p * L^q \subset L^r,
$$
compared to Euclidean space. Since $\HH^{n+1}$ can be viewed as $SO(n+1, 1)/SO(n+1)$, this has consequences for convolution on $\HH^{n+1}$. Anker and Pierfelice \cite{Anker-Pierfelice}, \cite[Section 4]{Anker-Pierfelice-2014} showed that convolution with a radial kernel $\kappa(r)$ satisfies
\begin{equation}
\| f * \kappa \|_{L^q(\HH^{n+1})} \leq C_q \| f \|_{L^{q'}(\HH^{n+1})} \bigg( \int_0^\infty (\sinh r)^n (1+r) e^{-nr/2} |\kappa(r)|^{q/2} \, dr \bigg)^{2/q}, \quad q \geq 2.
\label{KS}\end{equation}
From this we see that if $\kappa(r)$ is smooth and decays as $e^{-nr/2}$, then convolution with $\kappa$ maps $L^p$ to $L^{p'}$ for all $p \in [1, 2)$. Additionally, this non-Euclidean feature also affects the range of valid Strichartz estimates on (asymptotically) hyperbolic manifolds --- see \cite{Anker-Pierfelice, Ionescu-Staffilani-Mathann-2009, Chen-Hassell3}.
\end{remark}


\subsubsection{Spectral multipliers}

Our result for spectral multipliers is restricted to the case where the manifold is, in addition, a Cartan-Hadamard manifold, i.e. simply connected with nonpositive sectional curvatures.

\begin{theorem}\label{thm:sm} Suppose $(X, g)$ is an $n + 1$-dimensional non-trapping asymptotically hyperbolic manifold with no resonance at the bottom of spectrum. Suppose in addition that $X$ is simply connected with nonpositive sectional curvatures. Then for any $F \in H^{s}(\mathbb{R})$ supported in $[-1, 1]$ with $s > (n + 1)/2$, and for all $p \in [1, 2)$, $F(\alpha P)$ is a  bounded operator on $L^p + L^2$ uniformly with respect to parameter $\alpha$ for $0 < \alpha < 1$, in the sense $$ sup_{\alpha \in (0, 1]} \big\| F(\alpha P)\big\|_{L^p(X) + L^2(X) \longrightarrow L^p(X) + L^2(X)} < \infty.  $$\end{theorem}

This is weaker than Theorem~\ref{restriction2multiplier}, both because the function space is $L^p + L^2$ rather than $L^p$, but also because we have strengthened the Sobolev condition to $s > (n+1)/2$ for all $p$. From the perspective of harmonic analysis, it would be interesting to find a `better' function space, that is, more closely associated to the Laplacian, to accommodate the boundedness of the spectral multiplier. Modern harmonic analysis (Calder\'{o}n-Zygmund theory) is generally built on spaces with a doubling measure, which activates some kind of covering lemma and gives a simple structure of cube nets. Though some authors have investigated non-doubling spaces, the advances are mainly restricted to spaces of polynomial growth, which are ``semi-doubling".  In any case, the harmonic analysis on space of exponential growth is barely explored. One recent work along these lines is due to Bouclet \cite{Bouclet-ann Fourier-2011}, where it is shown that semiclassical spectral multipliers are bounded on appropriate \emph{weighted} $L^p$ spaces in a setting with exponential volume growth. 
The authors plan to pursue this question in future publications. 
\subsection{Strichartz estimates on asymptotically hyperbolic manifolds}

In the third paper in this series, \cite{Chen-Hassell3}, the first author will prove global-in-time Strichartz type estimates without loss  on non-trapping asymptotically hyperbolic manifolds. Namely,
for solutions of the inhomogeneous Schr\"odinger equation,
\begin{equation*}\label{schrodinger cauchy} \left\{ \begin{array}{c} i \frac{\partial}{\partial t} u  +  \Delta u = F(t, z)  \\  u(0, z) = f(z) \end{array}
\right. \end{equation*} with $f$ and $F$ orthogonal to eigenfunctions of $\Delta$ 
on an $n + 1$-dimensional asymptotically hyperbolic manifold $X$, one has the estimate
$$\|u\|_{L^p(\mathbb{R}, L^q(X))} \leq C \|f\|_{L^2(X)} + \|F\|_{L^{\tilde{p}^\prime}(\mathbb{R}, L^{\tilde{q}^\prime}(X))}$$
provided the pairs $(q, r)$ and $(\tilde{q}, \tilde{r})$ are  hyperbolic Schr\"{o}dinger admissible pairs of exponents.

\subsection{Outline of the paper}

The paper is organized as follows. In Section~\ref{sec:model}, we show how the main results in Section~\ref{subsec:mr} follow in the simple case of hyperbolic 3-space $\HH^3$.
In Section~\ref{sec:ahm}, we review the geometry and analysis of asymptotically hyperbolic manifolds, recalling the main results of \cite{Mazzeo-Melrose} and \cite{Chen-Hassell1}.
In Section~\ref{sec:le} we prove the restriction estimate, Theorem~\ref{thm:restriction},  for low energy, which exploits, in some sense, the Kunze-Stein phenomenon on $\HH^{n+1}$.

In Section~\ref{sec:wf}, in preparation for the high-energy estimates, we show how the microlocal support of the spectral measure may be localized by pre-and post-composition by pseudodifferential operators.
 In Section~\ref{sec:he} we prove Theorem~\ref{thm:restriction} for high energy. This uses, in a crucial way, the semiclassical Lagrangian structure of the high-energy spectral measure proved in \cite{Chen-Hassell1} and \cite{Wang}. Finally, in Section~\ref{sec:sm}, we prove the spectral multiplier result, Theorem~\ref{thm:sm}.

The authors would like to thank C. Guillarmou, A. McIntosh and A. Sikora for various helpful discussions during working on this paper. The authors gratefully acknowledge the support of the Australian Research Council through Discovery Grant  DP120102019.


\section{The model space $\mathbb{H}^3$}\label{sec:model}
In this section we illustrate the results of Theorems~\ref{thm:kernelbounds}, \ref{thm:restriction} and \ref{thm:sm} in the simple case of hyperbolic space. We focus on the case of $\HH^3$, in which the formulae are particularly simple.

Hyperbolic space can be defined in terms of the half space model
$$
\mathbb{H}^{n + 1} = \{ (x, y) \in \RR \times \RR^n \mid x > 0 \} ,
$$
equipped with the metric
$$\frac{dx^2 + dy^2}{x^2},$$
or in terms of the Poincar\'e disc model, as in \eqref{Poincare}.
For odd dimensions, that is, when $n = 2k$ is even,  the Schwartz kernel of $g(\sqrtdelta)$ is given by the explicit formula  \begin{equation}\label{funclapl}
\frac{1}{\sqrt{2\pi}} \bigg(- \frac{1}{2\pi}\frac{1}{\sinh(r)}\frac{\partial}{\partial r}\bigg)^k \hat{g}(r),
\end{equation} where $\sqrtdelta = (\Delta - {n}^2/4)^{1/2}$ as before,  and $r$ is geodesic distance on $\mathbb{H}^{n+1}$. See \cite[p.105]{Taylor2} for proof.

\subsection{Kernel bounds for the spectral measure}
In particular, $(\Delta - {n}^2/4 - \blambda^2)^{-1} = (P^2 -
\blambda^2)^{-1}$ for $\Im \blambda > 0$ is
\begin{equation}\begin{gathered}
-\frac{1}{2i\blambda}\bigg(-\frac{1}{2\pi}\frac{1}{\sinh(r)}\frac{\partial}{\partial r}\bigg)^k e^{i \blambda r}, \quad \Im \blambda > 0, \\
-\frac{1}{2i\blambda}\bigg(-\frac{1}{2\pi}\frac{1}{\sinh(r)}\frac{\partial}{\partial
r}\bigg)^k e^{-i \blambda r}, \quad \Im \blambda < 0.
\end{gathered}\label{ksm}\end{equation}
Setting now $k=1$, and applying Stone's formula \eqref{Stone2}, we find that  on $\HH^3$,
\begin{equation}
dE_{P}(\blambda) = \frac{\blambda}{2\pi}\frac{\sin (\blambda r)
}{\sinh r}. \label{spmeasH3}\end{equation}

\subsection{Restriction estimate}
Next, we deduce Theorem~\ref{thm:restriction} for $\HH^3$. The estimate for low energy follows immediately from \eqref{spmeasH3} and \eqref{KS}. The estimate for high energy and $p \in [1, 4/3]$ can be deduced from Theorem~\ref{kernel2restriction}:

\begin{proposition} $dE_{P}(\blambda)$ maps $L^p(\mathbb{H}^3)$ to $L^{p^\prime}(\mathbb{H}^3)$ with a bound $C \blambda^{3(1/p - 1/p^\prime) - 1}$ for all $\blambda > 0$, provided $1 \leq p \leq 4/3$.\end{proposition}

\begin{proof}

We assert the kernel estimates of Theorem \ref{kernel2restriction}
hold for this spectral measure, that is,
$$\big|dE_{P}(\blambda)\big| \leq C \frac{\blambda^2}{1 + \blambda
d(z, z^\prime)} \quad \mbox{and} \quad
\bigg|\bigg(\frac{d}{d\blambda}\bigg)^2\,dE_{P}(\blambda)\bigg| \leq
C \big(1 + \blambda d(z, z^\prime)\big).$$  In fact, one may see
$$dE_{P}(\blambda) = \frac{\blambda\sin\big(\blambda d(z,
z^\prime)\big)}{\sinh\big(d(z, z^\prime)\big)} \leq C \frac{\blambda}{d(z, z')} \leq
C \frac{\blambda^2}{1 + \blambda d(z, z^\prime)},$$ when $\blambda
d(z, z^\prime) > 1$; $$dE_{P}(\blambda) =
\frac{\blambda\sin\big(\blambda d(z, z^\prime)\big)}{\sinh\big(d(z,
z^\prime)\big)}\leq C \blambda^2 \leq C \frac{\blambda^2}{1 +
\blambda d(z, z^\prime)},$$ when $\blambda d(z, z^\prime) < 1$. On
the other hand, it is clear that
$$\bigg|\bigg(\frac{d}{d\blambda}\bigg)^2\,dE_{P}(\blambda)\bigg| =
\bigg|\frac{2d(z, z^\prime)\cos\big(\blambda d(z, z^\prime)
\big)}{\sinh\big(d(z, z^\prime)\big)} - \frac{ \blambda d(z,
z^\prime)^2 \sin\big(\blambda d(z, z^\prime)\big)}{\sinh\big(d(z,
z^\prime)\big)}\bigg| \leq C \big(1 + \blambda d(z,
z^\prime)\big).$$ Then applying Theorem \ref{kernel2restriction}
proves the proposition.

\end{proof}

In the range $p \in [4/3, 2)$ and for high energy, we again use complex interpolation, but rather than applying Theorem~\ref{kernel2restriction} as a black box, we need to modify the proof slightly. We observe that the spectral measure on $\HH^3$ satisfies
\begin{equation}
\Big| (\frac{d}{d\blambda})^j \spmeas \Big| \leq \blambda \quad
\text{ for all } j \geq 1.
\end{equation}
We substitute this estimate in place of the kernel bounds of
Theorem~\ref{kernel2restriction}, and run the proof of \cite[Section
3]{Guillarmou-Hassell-Sikora}. As in that proof, we consider the
analytic family of operators $\chi_+^a(\blambda - P)$. The proof
works just the same;\footnote{{We refer the reader to Section \ref{sec:restriction} for more details.}} in place of equation (3-7) of \cite[Section
3]{Guillarmou-Hassell-Sikora} and the previous equation, we obtain
\begin{equation}
\big\| \chi^{is}_+(\blambda - P) \big\|_{L^2 \to L^2} \leq C e^{\pi
|s|/2}
\end{equation}
on the line $\Re a = 0$, and
\begin{equation}
\big\| \chi^{-b+ is}_+(\blambda - P) \big\|_{L^1 \to L^\infty} \leq
C (1 + |s|) e^{\pi |s|/2} \blambda
\end{equation}
on the line $\Re a = -b$, for any $b > 1$. Let $p \in (4/3, 2)$, and
choose $b = p/(2-p)$. Using the fact that the spectral measure is
$\chi_+^{-1}(\blambda - P)$, and applying complex interpolation, we
find that
\begin{equation}
\big\| dE_P(\blambda) \big\|_{L^p \to L^{p'}} \leq C
\blambda^{p/(2-p)}.
\end{equation}

\subsection{Spectral multiplier estimate}

The hyperbolic space $\mathbb{H}^3$ is a non-doubling space but rather has exponential volume growth, i.e. the volume of a ball with radius $r$ satisfies $|B(r)| \sim (\sinh r)^2.$ The lack of doubling means that we cannot apply Theorem~\ref{restriction2multiplier} directly. Nevertheless, we can decompose the kernel of a spectral multiplier $F(\sqrtdelta)$ into two parts, one supported where $r \leq 1$ and one supported where $r \geq 1$, using a cutoff function $\chi_{\text{diag}}$, say, the characteristic function of $\{ r \leq 1 \}$ on $\HH^3 \times \HH^3$.

Then the proof of Theorem~\ref{restriction2multiplier} applies to $F(\sqrtdelta) \chi_{r\leq1}$, since all that is required for this proof to work is that doubling is valid for all balls of radius $\leq 1$, which is certainly true. We obtain

\begin{lemma}\label{smH3} For every even function $F \in H^s(\mathbb{R})$ supported in $[-1, 1]$ with $s > 3(1/p - 1/2)$, $F(\alpha \sqrtdelta)\chi_{r\leq1}$ maps $L^p(\mathbb{H}^3)$ to itself with a uniform bound $$\sup_{\alpha > 0} \|F(\alpha \sqrtdelta)\chi_{r\leq1}\|_{L^p(\mathbb{H}^3) \rightarrow L^p(\mathbb{H}^3)} \leq C \|F\|_{H^s},$$ provided $1 \leq p \leq 4/3$, where $\chi_{r\leq1}$ is the characteristic function of the set $\{(z_1, z_2) : d(z_1, z_2) < 1\}$.
\end{lemma}

In particular, if $s > 3/2$, then this is valid for $p=1$, and thus by interpolation and duality for all $p \in [1, \infty]$.

For the other part, supported where $r \geq 1$, we show boundedness from $L^p(\HH^3) \to L^2(\HH^3)$. By interpolation, it is enough to treat the case $p=1$, since boundedness $L^2 \to L^2$ follows immediately from the boundedness of $F$.

The $L^1 \to L^2$ operator norm of an integral operator $K(z_1, z_2)$ is bounded by
$$
\sup_{z_2} \bigg(\int \big|K(z_1, z_2)\big|^2\, d\mu_1\bigg)^{1/2}.
$$
We express the kernel of $F(\sqrtdelta)\chi_{r>1}$ using \eqref{ksm}. So we need to estimate
$$\int_{\mathbb{S}^2 \times [1, \infty)} \bigg|\frac{1}{{(2\pi)}^{3/2}}  \frac{1}{\sinh(r)}\frac{\partial}{\partial r} \hat{F}(r)\bigg|^2 \sinh^2(r) \,dr\,d\omega \leq C \int_1^\infty \bigg|\frac{\partial }{\partial r}\hat{F}(r)\bigg|^2\,dr.$$


Write $F_\alpha(\blambda) = F(\alpha \blambda)$. For any $\alpha >
0$, we get the estimate for $F_\alpha$:
\begin{eqnarray*}
\int_{\mathbb{S}^2 \times [1, \infty)}\bigg|\frac{\partial}{\partial r}\hat{F_\alpha}(r)\bigg|^2\,dr d\omega & = & C \int_1^\infty\bigg|\frac{\partial}{\partial r}\frac{\hat{F}(r/\alpha)}{\alpha}\bigg|^2\,dr\\
& \leq & C \frac{1}{\alpha^3} \int_{1/\alpha}^\infty\bigg|\frac{\partial}{\partial r} \hat{F}(r)\bigg|^2\,dr\\
& \leq & C  \int_{1/\alpha}^\infty r^{3}  \bigg|\frac{\partial}{\partial r} \hat{F}(r)\bigg|^2\,dr\\
& \leq & C \big\| \blambda F(\blambda) \big\|_{H^{3/2}}^2 \leq  C
\big\|  F\big\|_{H^{3/2}}^2
\end{eqnarray*}
using the compact support of $F$.
Combining this estimate with Lemma~\ref{smH3},   we have proved Theorem~\ref{thm:sm} in the case of $\HH^3$.

\section{The geometry and analysis of asymptotically hyperbolic manifolds}\label{sec:ahm}

\subsection{$0$-structure}

Suppose $(X^\circ, g)$ is an $(n + 1)$-dimensional asymptotically hyperbolic manifold. Let $X$ be the compactification. We write $x$ for a boundary defining function, and use local coordinates $(x, y_1, \dots, y_n)$ near a boundary point of $X$, where $y = (y_1, \dots, y_n)$ restrict to coordinates on $\partial X$, or $z = (z_1, \dots, z_{n+1})$ in the interior of $X$.

Consider the space of smooth vector fields on the compactification, $X$, that are of uniformly finite length. Due to the factor $x^{-2}$ in the metric, such vector fields take the form $xV$, where $V$ is a smooth vector field on $X$. Such vector fields
are called $0$-vector fields, spanned over $C^\infty(X)$ near the boundary by the vector fields $x \partial_x$ and $x \partial_{y_i}$, $1 \leq i \leq n$. As observed by Mazzeo-Melrose, they are the space of sections of a vector bundle, known as the $0$-tangent bundle, ${}^0 TX$.

The dual bundle, known as the $0$-cotangent bundle and denoted ${}^0 T^* X$, is spanned by local sections $dx/x$ and $dy_i/x$ near the boundary. It follows that, near the boundary of $X$, we can write points $q \in {}^0 T^* X$ in the form
\begin{equation}
q = \lambda \frac{dx}{x} + \sum_{j=1}^n \mu_j  \frac{dy_j}{x};
\end{equation}
this defines linear coordinates $(\lambda, \mu)$ on each fibre of ${}^0 T^* X$ (near the boundary), depending on the coordinate system $(x, y)$.

The Laplacian $\Delta$ on $X$ is built out of an elliptic combination of $0$-vector fields. In fact, in local coordinates $(x, y)$ near the boundary of $X$, with $g$ taking the form \eqref{metric},  it takes the form
$$
(x D_x)^2 + i n x D_x + (x D_{y_i}) h^{ij} (x D_{y_j}) \text{ modulo } x  \ {}^0 \mathrm{Diff}^1(X),
$$
where we use ${}^0 \mathrm{Diff}^k(X)$ to denote differential operators of order $k$ generated over $C^\infty(X)$ by $0$-vector fields.

\subsection{The $0$-double space}
We would like to understand the nature of the Schwartz kernel of the resolvent $(\Delta - \zeta(n - \zeta))^{-1}$, on $X^\circ \times X^\circ$. Following Mazzeo-Melrose, we use a  compactification of the double space $X^\circ \times X^\circ$ that reflects the geometry of $(X^\circ, g)$, particularly near the diagonal. This is important as we want to view the resolvent as some sort of pseudodifferential operator, which means that we need a precise notion of what it means for a distribution to be conormal to the diagonal, uniformly out to infinity. 

Compactifying $X^\circ$ to $X$, we can initially view the resolvent kernel on $X^2$. However, on this space, the diagonal is not a p-submanifold where it meets the boundary. That is, near the boundary of the diagonal in $X^2$, there are no local coordinates of the form $(x, x', w)$ where $x$, resp. $x'$ is a boundary defining function for the left, resp. right, copy of $X$ and $w$ are the remaining coordinates, such that the diagonal is given by the vanishing of a subset of these coordinates. To give a workable definition of conormality to a submanifold, we require it to be a  p-submanifold. To remedy this, we blow up (in the real sense) the boundary of the diagonal. This creates a manifold with corners, denoted $X^2_0$, the `$0$-double space', with three boundary hypersurfaces: the two original ones, $\FL$ `left face' and $\FR$ `right face', corresponding to $\{ x = 0 \}$ and $\{ x' = 0 \}$ in $X^2$, and the new face $\FF$, the `front face', created by blowup --- see Figure~\ref{figure:X20}. We denote a generic boundary defining function for $\FL, \FR$ or $\FF$ by $\rho_{\FL}, \rho_{\FR}$ and $\rho_{\FF}$, respectively. 

\begin{center}\begin{figure}
\includegraphics[width=0.6\textwidth]{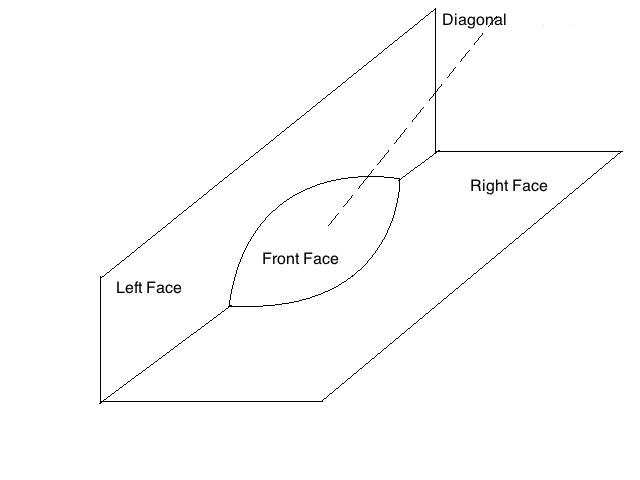}\caption{\label{fig}The $0$-blown-up double space $X \times_0 X$}
\label{figure:X20}\end{figure}\end{center}


As in \cite{Chen-Hassell1}, we write down coordinate systems in various regions of $X^2_0$, in terms of coordinates $(x,y) = (x, y_1, \dots, y_n)$ near the boundary of $X$, or $z = (z_1, \dots, z_{n+1})$ in the interior of $X$.
The unprimed coordinates always indicate those lifted from the left factor of $X$, while primed coordinates indicate those lifted from the right factor.
We label these different regions as follows:
\begin{itemize}
\item{Region 1:} In the interior of $X^2_0$. Here we use coordinates $$(z, z') = (z_1, \dots, z_{n+1}, z'_1, \dots, z'_{n+1}).$$

\item{Region 2a:} Near $\FL$ and away from $\FF$ and $\FR$.
In this region, we use $(x, y, z')$.
\item{Region 2b:} Near $\FR$ and away from $\FF$ and $\FL$.
Symmetrically,  we use $(z, x', y')$.

\item{Region 3:} Near $\FL \cap \FR$ and away from $\FF$.
Here we use $(x, y, x', y')$.

\item{Region 4a:} Near $\FF$ and away from $\FR$.
This is near the blowup. In this region we can use $s = x/x'$ for a boundary defining function for $\FF$. We use coordinate system
$$
s = \frac{x}{x'}, \ x', \ y, \ Y = \frac{y' - y}{x'}.
$$
\item{Region 4b:} Near $\FF$ and away from $\FL$.
Symmetrically, we use
$$
s' = \frac{x'}{x}, \ x, \ y', \ Y' = \frac{y - y'}{x}.
$$
\item{Region 5:} Near the triple corner $\FL \cap \FF \cap \FR$.
In this case, a boundary defining function for $\FF$ is $|y'-y|$. By rotating the $y$ coordinates, we can assume that $|y'_1 - y_1| \geq c |y' - y|$ in a neighbourhood of any given point in the triple corner. Assuming this, we use coordinates
$$
s_1 = \frac{x}{y_1' - y_1}, \ s_2 = \frac{x'}{y_1' - y_1}, \ t = y_1' - y_1, \ Z_j = \frac{y_j^\prime - y_j}{y_1^\prime - y_1} \, (j > 1).
$$
\end{itemize}

On $X^2_0$, the lift of the diagonal, denoted $\diagz$, meets the boundary in the interior of the front face $\FF$. It has several good geometric properties:
\begin{itemize}
\item $\diagz \subset X^2_0$ is a p-submanifold disjoint from $\FL$ and $\FR$;

\item the $0$-vector fields $x \partial_x$, $x \partial_{y_i}$ lift from the left and right factors of $X$ to be vector fields on $X^2_0$ that
are non-tangential to $\diagz$, uniformly down to the boundary of $\diagz$. Moreover, these vector fields span the normal bundle of $\diagz$, again uniformly down to the boundary.

\item The distance function $d(z,z')$ is smooth in a deleted neighbourhood of $\diagz$, and its square is a quadratic defining function for the lifted diagonal, i.e. it is smooth and vanishes to precisely second order at $\diagz$.
\end{itemize}

\subsection{Resolvent kernel} \label{resolvent kernel}

Taking advantage of the first and second geometric properties listed above, Mazzeo and Melrose `microlocalized' the space of $0$-differential operators to a calculus of $0$-pseudodifferential operators on $X$. The set of pseudodifferential operators of order $m$ on $X$, denoted $ \Psi_0^m(X)$,  is,  by definition, the set of operators on half-densities, whose Schwartz kernels are conormal of order $m$ to $\diagz$, and vanish to infinite order at $\FL$ and $\FR$.

Mazzeo and Melrose \cite{Mazzeo-Melrose} showed that the resolvent 
$$
R(\blambda) = \big(\Delta - n^2/4 - \blambda^2 \big)^{-1}, \quad \Im \blambda < 0, 
$$  takes the form
\begin{equation}
R(\blambda) \in \Psi^{-2}_0(X) + \rho_L^{n/2 + i\blambda} \rho_R^{n/2 + i\blambda} C^\infty(X \times_0 X).
\label{leres}\end{equation}

For low energy, this description is precise enough to deduce kernel
estimates for the spectral measure, restriction estimates, and
spectral multiplier theorems. However, as $\blambda \to \infty$, we
need a uniform description of the resolvent, and in particular we
need to understand its oscillatory nature. For this, we use the
description by Melrose-Sa Barreto-Vasy \cite{Melrose-Sa
Barreto-Vasy}, Wang \cite{Wang} and the present authors
\cite{Chen-Hassell1} (in the first paper of this series) of the
high-energy resolvent as a semiclassical Lagrangian distribution.
This is associated to the bicharacteristic relation on $X^\circ
\times X^\circ$, that is, the submanifold of $T^* X^\circ \times T^*
X^\circ$ given by
$$
\BR = \{ (z, \zeta; z', - \zeta') \mid |\zeta|_g = |\zeta'|_g = 1, \ (z, \zeta) \text{ and } (z', \zeta') \text{ lie on the same bicharacteristic} \},
$$
which is a smooth Lagrangian submanifold provided that $X$ is nontrapping.
By `bicharacteristic' we mean here the integral curves of the symbol of $\Delta$ on the set where $\sigma(\Delta) = 1$. In this case these are precisely geodesics, viewed as living in the cotangent bundle.

The bicharacteristic relation splits into the forward and backward bicharacteristic relations, $\BR_+$ and $\BR_-$, which\footnote{The forward bicharacteristic relation $\BR_+$ was denoted $\FBR$ in \cite{Chen-Hassell1}.} consist of those points $(z, \zeta; z', - \zeta') \in \BR$ for which $(z, \zeta)$ is on the forward/backward half of the bicharacteristic relative to $(z', \zeta')$. These two halves meet at $\BR \cap N^* \diag$, where
$N^* \diag$ denotes the conormal bundle of the diagonal,
$$
N^* \diag = \{ (z, \zeta, z', -\zeta) \}.
$$

We wish to understand the way in which $\BR$ compactifies when viewed as living over the double space $X^2_0$. We consider the bundle ${}^\Phi T^* X^2_0$, obtained by pulling back the bundle $({}^0 T^* X)^2$ to $X^2_0$ by the blowdown map $\beta : X^2_0 \to X^2$. We denote the bundle projection maps by ${}^\Phi \pi : {}^\Phi T^* X^2_0 \to X^2_0$. Then, as explained in \cite[Section 3]{Chen-Hassell1}, it is convenient to `shift' $\BR$ by the map $T_\pm$ defined by
\begin{equation}
T_\pm(q) = q \mp d(\log \rho_L) \mp d(\log \rho_R),  \quad q \in {}^\Phi T^* X^2_0,
\label{tL1}\end{equation}
for some choice of boundary defining functions $\rho_L$ for $\FL$ and $\rho_R$ for $\FR$; that is, we consider $T_-^{-1}(\BR_-) \cup T_+^{-1}(\BR_+)$. It is convenient here to assume that $\rho_L$ and $\rho_R$ are both constant near $\diagz$, so that these two shifted Lagrangian join smoothly at $N^* \diagz$.

In \cite{Chen-Hassell1} we showed\footnote{This was shown for the forward bicharacteristic relation in \cite{Chen-Hassell1}, but the statements in Proposition~\ref{prop:flowout-structure} follow immediately.}

\begin{proposition}\label{prop:flowout-structure}
The bicharacteristic relation $\BR$
can be expressed as the union of two relatively open subsets $\BR^{nd} \cup \BR^*$ , having the following properties.
\begin{itemize}
\item $\BR^{nd}$ contains a neighbourhood of the intersection $\BR \cap N^* \diag$ in $\BR$, that is, the points $(z, \zeta, z, -\zeta) \in \BR$.
\item Let $\Lambda^{nd}$ denote the lift of $\BR^{nd}$ to ${}^\Phi T^* X^2_0$, together with its limit points lying over $\FF$, $\FL$ and $\FR$. 
Let $\Lambda^{nd}_{\pm} = \Lambda^{nd} \cap \BR_\pm$ denote the two halves of this submanifold,   meeting at $N^* \diagz$. Then $\Lambda^{nd}_{\pm}$
are  manifolds with codimension three corners, with the property that the interior of $\Lambda^{nd}_+ $ is the 
graph of the differential of the distance function on some deleted neighbourhood $V$ of $(\diagz \cup \FF)\subset X^2_0$, and the interior of $\Lambda^{nd}_- $ is the 
graph of minus the differential of the distance function on  $V$.   Thus the projection ${}^\Phi \pi : \Lambda^{nd} \to X^2_0$ has full rank restricted to $\Lambda^{nd}$, except at $\Lambda^{nd} \cap N^* \diagz = \Lambda^{nd}_+ \cap \Lambda^{nd}_-$, where the rank of the projection ${}^\Phi \pi : \Lambda^{nd} \to X^2_0$ drops by $n$.
The boundary hypersurfaces of $\Lambda^{nd}$ are  $\partial_{\FF}\Lambda^{nd}$, lying over $\FF$,  $\partial_{\FL}\Lambda^{nd}$, lying over $\FL$ and $\partial_{\FR}\Lambda^{nd}$, lying over $\FR$. 

\item The image $\tL^{nd}$ of $\Lambda^{nd}$ under the shift \eqref{tL1}  is a smooth Lagrangian submanifold of $T^* X^2_0$ (NB: the standard cotangent bundle, not ${}^\Phi T^* X^2_0$) with codimension three corners. The projection $\pi : T^* X^2_0 \to X^2_0$ restricts to a map $\tL^{nd} \to X^2_0$ with full rank, except at $\tL^{nd} \cap N^* \diagz = \tL^{nd}_+ \cap \tL^{nd}_-$, where the rank of the projection ${}^\Phi \pi : \tL^{nd} \to X^2_0$ drops by $n$.

\item Let $\widetilde{\BR^*}$ denote the image of $\BR^*$ under the shift \eqref{tL1}, and let $\tL^*$ denote the closure of $\widetilde{\BR^*}$ in $T^* X^2$. Then $\tL^*$
 is a smooth Lagrangian submanifold of $T^* X^2$ (NB: the standard cotangent bundle, not the $0$-cotangent bundle)  with codimension two corners.
\end{itemize}

\end{proposition}

In terms of these Lagrangian submanifolds we determined the semiclassical nature of the resolvent kernel in \cite[Theorem 38]{Chen-Hassell1}. In view of Stone's formula, \eqref{Stone},  this has the (almost) immediate consequence for the spectral measure:

\begin{theorem}\label{thm:specmeas} Let $(X^\circ, g)$ be an asymptotically hyperbolic non-trapping manifold, with no resonance at the bottom of the continuous spectrum. Then the spectral measure $\spmeas$ with $\blambda = 1/h$ can be expressed as a sum of the following terms:
\begin{itemize}
\item[(i)] A semiclassical Lagrangian distribution
in  $(\rho_L \rho_R)^{n/2} I^{-1/2}(X^2_0, \Lambda^{nd}; \Omegazh)$, where $\Lambda^{nd}$ is as in Proposition~\ref{prop:flowout-structure}.
\item[(ii)] an element of
$$(\rho_L \rho_R)^{n/2-i/h} h^\infty C^\infty(X^2_0\times [0, h_0]; \Omegazh) + (\rho_L \rho_R)^{n/2+i/h} h^\infty C^\infty(X^2_0\times [0, h_0]; \Omegazh),$$
which can be regarded as an element of type (i) of order $-\infty$.
\item[(iii)] a kernel lying in $$(x x')^{n/2-i/h} I^{-1/2}(X^2, \tL_{+}^*, \Omegazh) + (x x')^{n/2+i/h} I^{-1/2}(X^2, \tL_-^*, \Omegazh),$$ also associated to the bicharacteristic flowout, as above, but living on $X^2$ rather than $X^2_0$;
\item[(iv)] an element of
$$(xx')^{n/2-i/h} h^\infty C^\infty(X^2\times [0, h_0]; \Omegazh) + (xx')^{n/2+i/h} h^\infty C^\infty(X^2\times [0, h_0]; \Omegazh),$$
which can be regarded as an element of type (iii) of order $-\infty$.
\end{itemize}
 \end{theorem}

 \begin{proof} We first remark that the change in order from $+1/2$ for the resolvent in \cite[Theorem 38]{Chen-Hassell1} to $-1/2$ for the spectral measure is simply due to the fact that the semiclassical resolvent in \cite{Chen-Hassell1} is $h^{-2}$ times the resolvent in \eqref{Stone2}, together with the factor of $\blambda = h^{-1}$ in \eqref{Stone2}.

 In \cite[Theorem 38]{Chen-Hassell1} it was shown that the resolvent kernel has a similar, but slightly more complicated structure: in place of the first term above, it consists of a semiclassical pseudodifferential operator, together with a semiclassical intersecting Lagrangian distribution associated to $N^* \diagz$ together with the forward/backward half of the bicharacteristic relation (for the outgoing/incoming resolvent). We claim that when the incoming resolvent is subtracted from the outgoing, the pseudodifferential part cancels, and what is left is a Lagrangian distribution associated to the full bicharacteristic relation. This follows since the spectral measure satisfies an elliptic equation
 $$
 (h^2 \Delta - h^2 n^2/4 - 1) \spmeas = 0, \quad h = \blambda^{-1}.
 $$
 Therefore, the spectral measure can have no semiclassical wavefront set outside the zero set of the symbol of $h^2 \Delta - 1$. This excludes all of $N^* \diagz$ except for that part contained in $\BR$. In addition, propagation of Lagrangian regularity\footnote{Propagation of Lagrangian regularity is the statement that, if $P$ is an operator of real principal type, $Pu = O(h^\infty)$, and $u$ is a Lagrangian distribution microlocally in some region $V$ of phase space, then $u$ is also Lagrangian along the bicharacteristics passing through $V$. It follows from the parametrix construction for Lagrangian solutions of operators of real principal type, and the propagation of singularities theorem.} shows that the spectral measure is a Lagrangian distribution across $N^* \diagz$ (given that we already know that it is Lagrangian on both sides of $N^* \diagz$ corresponding to forward and backward flowout, and given that the Hamilton vector field of the symbol does not vanish at $\BR \cap N^* \diagz$). This concludes the proof.
 \end{proof}

 \subsection{The distance function on $X^2_0$} The distance function on $X^2_0$ satisfies  
 
 \begin{proposition}\label{prop:dist}
On $X^2_0$, the Riemannian distance function $d(z, z')$ is given by
$$
d(z, z') = -\log(\rho_L \rho_R) + b(z, z'),
$$
where $b(z, z')$ is uniformly bounded on $X^2_0$.
\end{proposition}

\begin{remark} The result in the case that $(X^\circ, g)$ is a small perturbation of $(\HH^{n+1}, g_{hyp})$ was shown by Melrose, S\'{a} Barreto and Vasy \cite[Section 2]{Melrose-Sa Barreto-Vasy}. 
\end{remark}

\begin{proof}
 Consider two points $p, p' \in X^\circ$. When $(p,p')$ are in a sufficiently small neighbourhood $U$ of the front face $\FF$, say $p = (x, y)$, $p' = (x', y')$ with $x, x' < \epsilon$
 and $d(y, y') < 4\epsilon$ (taken with respect to the metric $h(0)$ at the boundary),
 then the distance function parametrizes the Lagrangian $\Lambda^{nd}$, and it follows from \cite[Proposition 20]{Chen-Hassell1} that this takes the form $-\log(\rho_L \rho_R) + C^\infty(X^2_0)$ in a neighbourhood of $\FF$.

Define $K \subset X^\circ$ to be the compact set $\{ x \geq \epsilon \}$. Let $M $ be the diameter of $K$, that is, the maximum distance between two points of $K$.

Now suppose that $(p, p') \notin U$. In the complement of $U$, we can take $\rho_L = x$ and $\rho_R = x'$.

If both $p$ and $p'$ lie in $K$, then the distance between $p$ and $p'$ is at most $M$, hence $|d(p, p') + \log (\rho_L \rho_R)| \leq M + 2\max_K |\log x| = O(1)$.

If one point, say $p$, lies in $K$ and $p'$ is not in $K$, then a lower bound on $d(p,p')$ is the distance from $p'$ to the boundary of $K$, which is exactly $\log \epsilon - \log x = -\log (\rho_L \rho_R) + O(1)$. On the other hand, an upper bound is the length of the path from $p'$ to the closest point $p''$ on $\partial K$, plus the distance from $p''$ to $p$. This is at most $\log \epsilon - \log x +M = -\log (\rho_L \rho_R) + O(1)$.

If neither point lies in $K$, then write $p = (x, y)$ and $p' = (x', y')$.
Due to the definition of $U$, we must have $d(y, y') \geq 4\epsilon$. We claim that any geodesic between $p$ and $p'$ must enter $K$. It follows from this claim that a lower bound on the distance between $p$ and $p'$ is the distance from $p$ to $\partial K$ plus the distance between $p'$ to $\partial K$, which is $-\log x - \log x' + 2\log \epsilon$, that is, $-\log (\rho_L \rho_R) + O(1)$. Also, an upper bound on the distance is clearly $-\log x - \log x' + 2\log \epsilon + M$ which is also $-\log (\rho_L \rho_R) + O(1)$. Thus, to complete the proof, it remains to establish the claim above.

Consider any geodesic that lies wholly within the region $x \leq \epsilon$. Parametrize the geodesic with arc length, such that the value of $x$ is maximal at $t=0$ --- say, equal to $x_{max} \leq \epsilon$. We recall the geodesic equations for $(x, y, \lambda, \mu)$ where these are the $0$-cotangent variables as described in \cite[Section 2]{Chen-Hassell1}:
\begin{equation}
\left\{ \begin{array}{ccl}
\dot{x} &=& x \lambda  \\
\dot{y_i} & = & x h^{ij} \mu_j  \\
\dot{\lambda} & = & -\Big(h^{ij} +  \frac1{2}x \partial_x h^{ij}  \Big) \mu_i \mu_j    \\
\dot{\mu_i} & = & \Big( \lambda\mu_i - \frac1{2}(x \partial_{y_i}  h^{jk}) \mu_j \mu_k \Big) \end{array} \right..
\label{ge}\end{equation}
We also recall that $\lambda^2 + |\mu|^2 = 1$ along the geodesic, where $|\mu|^2 = h^{ij}(x, y)\mu_i \mu_j$. We see that
$$
\dot \lambda = -|\mu|^2 (1 + O(x)) = -(1-\lambda^2)(1 + O(x)).
$$
Thus, we have
\begin{equation}
\dot \lambda \leq -\alpha(1 - \lambda^2), \quad \lambda(0) = 0.
\label{lambda-eqn}\end{equation}
for some $\alpha \sim 1 + O(\epsilon)$ slightly less than $1$, which can be taken as close as desired to $1$ by choosing $\epsilon$ sufficiently small. The initial condition $\lambda(0) = 0$ arises as $\dot x = 0$ at $t=0$.

We can integrate the differential inequality \eqref{lambda-eqn} to obtain
$$
\frac1{2} \int \Big( \frac1{1 + \lambda} + \frac1{1 - \lambda} \Big) \, d\lambda \leq - \alpha \int dt,
$$
which yields
$$
\lambda(t) \leq - \frac{1 - e^{-2\alpha t}}{1 + e^{-2\alpha t}}.
$$
Plugging this into the equation for $x$, we find that
$$
\dot x \leq -x \frac{1 - e^{-2\alpha t}}{1 + e^{-2\alpha t}}.
$$
Integrating this, we find that
$$
\log x \leq - \int \frac{1 - e^{-2 \alpha t}}{1 + e^{2\alpha t}}  e^{2\alpha t} \, dt,
$$
and with the help of the substitution $v = e^{2\alpha t}$, we obtain
$$
x \leq x_{max} \Big( \frac{2 e^{\alpha t}}{1 + e^{2\alpha t}} \Big)^{1/\alpha}.
$$
Finally we turn to the equation for $y$. We have
$$
|\dot y| = x |\mu| = x \sqrt{1 - \lambda^2} \leq x_{max} \Big( \frac{2 e^{\alpha t}}{1 + e^{2\alpha t}} \Big)^{1+1/\alpha}.
$$
Integrating the RHS from $0$ to $\infty$ at the value $\alpha = 1$ gives $x_{max}$. We get the same result for negative time, so that means that, along this geodesic, the maximum distance that $y$ can travel, with respect to the $h(0)$ metric, is
$$
2x_{max} \int_0^\infty \Big( \frac{2 e^{\alpha t}}{1 + e^{2\alpha t}} \Big)^{1+1/\alpha} \, dt.
$$
This is equal to $2x_{max}$ when $\alpha = 1$, and  depends continuously on $\alpha$, hence is close to
$2x_{max}$ for $\alpha$ close to $1$, that is, when $\epsilon$ is sufficiently small\footnote{As a check, we note that for the hyperbolic metric on the upper half space, where the geodesics are great circles on planes perpendicular to the boundary and centred on the boundary, the maximum distance is indeed $2 x_{max}$.}. It follows that if $d(y, y') \geq 4\epsilon$, the geodesic between $p$ and $p'$ must enter the region $\{ x \geq \epsilon \}$ (provided $\epsilon$ is sufficiently small). This completes the proof of the proposition.
\end{proof}

\section{Low energy behaviour of the spectral measure}\label{sec:le}

Pointwise bounds on the spectral measure, and restriction estimates, are readily deduced from the regularity statement \eqref{leres} for the low energy resolvent.

\subsection{Pointwise bounds on the spectral measure}

The regularity statement \eqref{leres} for the resolvent, together with Stones's formula \eqref{Stone2}, implies that the Schwartz kernel of the low energy spectral measure $\spmeas$ takes the form
\begin{equation}
\blambda \Big( (\rho_L \rho_R)^{n/2 + i\blambda} a(\blambda) -
(\rho_L \rho_R)^{n/2 - i\blambda} a(-\blambda) \Big),
\end{equation}
where $a(\blambda)$ is a $C^\infty$ function on $X^2_0$ depending
holomorphically on $\blambda$ for small $\blambda$. Here we use our
assumption that the resolvent is holomorphic in a neighbourhood of
$n^2/4$, the bottom of the essential spectrum; on the other hand,
the nontrapping assumption is irrelevant here.

We write the RHS as
\begin{equation}\begin{gathered}
\blambda \Big((\rho_L \rho_R)^{n/2 + i\blambda} - (\rho_L
\rho_R)^{n/2 - i\blambda}\Big) a(0) \\ + \blambda \Big( (\rho_L
\rho_R)^{n/2 + i\blambda} \big(a(\blambda) - a(0) \big) - (\rho_L
\rho_R)^{n/2 - i\blambda} \big( a(-\blambda) - a(0) \big)  \Big),
\end{gathered}\end{equation}
which implies that the kernel is bounded pointwise by
$$ C \blambda (\rho_L \rho_R)^{n/2} \big| \sin (\blambda \log(\rho_L \rho_R)) \big| + C' \blambda^2  (\rho_L \rho_R)^{n/2}.
$$
Using Proposition~\ref{prop:dist} we may write $|\log (\rho_L \rho_R)| = d(z, z') + O(1)$. Then estimating  the sine factor by  $|\sin s \,| \leq |s| (1+|s|)^{-1}$, we obtain the low energy ($\blambda \leq 1$) estimate in Theorem~\ref{thm:kernelbounds}.

\subsection{Restriction estimate}

We have just seen that the spectral measure for low energy,
$\blambda \leq 1$, is bounded pointwise by $\blambda^2$ times $-\log
(\rho_L \rho_R) (\rho_L \rho_R)^{n/2}$. Thus, to prove the low
energy restriction estimate, it suffices to show that an integral
operator, say $A(z,z')$, with kernel bounded pointwise by $-\log
(\rho_L \rho_R) (\rho_L \rho_R)^{n/2}$ maps $L^p(X)$ to $L^{p'}(X)$
for all $p \in [1, 2)$.

To do this, we break up the kernel $A(z,z')$ into pieces. Let $U$ be a neighbourhood of the front face $\FF$ in $X^2_0$. We consider $A(z,z') 1_U$ and $A(z,z') 1_{X^2_0 \setminus U}$ separately.

First consider $A(z,z') 1_{X^2_0 \setminus U}$. In this region, we may take $\rho_L = x$ and $\rho_R = x'$.  This part of the kernel is therefore bounded by $C(-\log x) x^{n/2} (-\log x') {x'}^{n/2}$. Thus, it is easy to check that $A(z,z') 1_{X^2_0 \setminus U}$ is in $L^{p'}(X \times X)$, for any $p' > 2$. It therefore maps $L^p(X)$ to $L^{p'}(X)$ for all $p \in [1, 2)$.

Now consider the remainder of the kernel, $A(z,z') 1_U$. We may further decompose the set $U$ into subsets $U_i$, where on each $U_i$, we have $x \leq \epsilon, x' \leq \epsilon$ and $d(y, y_i), d(y', y_i) \leq \epsilon$ for some $y_i \in \partial X$ (where the distance is measured with respect to the metric $h(0)$ on $\partial X$). Choose local coordinates $(x, y)$ on $X$, centred at $(0, y_i) \in \partial X$, covering the set $ V_i = \{ x \leq \epsilon, d(y, y_i) \leq \epsilon \}$, and use these local coordinates to define a map $\phi_i$ from $V_i$ to a neighbourhood $V_i'$ of $(0,0)$ in hyperbolic space $\HH^{n+1}$ using the upper half-space model (such that the map is the identity in the given coordinates).

The map $\phi_i$ induces a diffeomorphism $\Phi_i$ from $U_i \subset X^2_0$ to a subset of $(B^{n+1})^2_0$, the double space for $\HH^{n+1}$, covering the set $x \leq \epsilon, x' \leq \epsilon, |y|, |y'| \leq \epsilon$ in this space. Clearly, this map identifies $\rho_L$ and $\rho_R$ on $U_i$ with corresponding boundary defining functions for the left face and right face on $(B^{n+1})^2_0$. We now consider the kernel
\begin{equation}
\phi_i \circ A 1_{U_i} \circ \phi_i^{-1}
\label{transfer}\end{equation} as an integral operator on
$(B^{n+1})^2_0$. This kernel is bounded by $(1+r) e^{-nr/2}$, where
$r$ is the geodesic distance on $\HH^{n+1}$, since $(1+r)$ is
comparable to $-\log(\rho_L \rho_R)$ on $(B^{n+1})^2_0$. Therefore,
using \eqref{KS}, \eqref{transfer} is bounded from $L^p(\HH^{n+1})$
to $L^{p'}(\HH^{n+1})$ for every $p \in [1, 2)$. It is clear that
$\phi_i$ are bounded, invertible maps from $L^p(V_i)$ to
$L^p(V'_i)$. This shows that the kernel $A 1_{U_i}$ is bounded from
$L^p(V_i)$ to $L^{p'}(V_i)$ for all $p \in [1, 2)$. This completes
the proof of Theorem~\ref{thm:restriction} in the case of low
energy, $\blambda \leq 1$.

\section{Pseudodifferential operator microlocalization}\label{sec:wf}

According to Theorem~\ref{thm:specmeas},
the spectral measure is a Lagrangian distribution associated to the
Lagrangian submanifold $\Lambda^{nd}$ (on ${}^\Phi T^*X^2_0$) and to the Lagrangian submanifold $\Lambda^*$ (on ${}^0 T^* X^2$).  We first define the notion of microlocal support, which is a  closed subset of ${}^\Phi T^* X^2_0$ giving the essential support `in phase space', for such distributions. It is a special case of the notion of semiclassical wavefront set, defined for example in \cite[Section 8.4]{Zworski}.
We consider a local oscillatory integral expression for $u \in I^m(\Lambda)$, where $\Lambda$ is a Lagrangian submanifold of ${}^\Phi T^* X^2_0$. This is given by a local expression
\begin{equation}
u = h^{-m - (n+1)/2 - k/2} \int e^{i\phi(Z, v)/h} a(Z, v, h) \, dv +O(h^\infty),
\label{localpar}\end{equation}
where $v \in \RR^k$, with $a$ smooth, and we use $Z$ for local coordinates on $X^2_0$, as described explicitly in Regions 1--5 in Section~\ref{sec:ahm}. This requires that $\phi$ locally parametrizes $\Lambda$ (nondegenerately), i.e. the map $\iota$  from $C_\phi$,
$$
C_\phi = \{ (Z, v) \mid d_v \phi(Z, v) = 0 \}
$$
to $\Lambda$, given by
$$
C_\phi \ni (Z, v) \mapsto \iota (Z, v) := (Z, d_Z \phi(Z, v)) \in \Lambda,
$$
is a local diffeomorphism. The microlocal support $\WF(u)$ of \eqref{localpar} is then contained in $\Lambda$ (in general it can be any ), and is determined by the support of the amplitude $a$:
\begin{equation}
\WF(u) = \{ q \in \Lambda \mid a(Z, v, h) \text{ is not $O(h^\infty)$ in a neighbourhood of } (Z, v, 0), \text{ where } \iota(Z, v) = q \}.
\end{equation}
It depends only on $u$, not the particular form of \eqref{localpar}. 

We also recall that the Schwartz kernel of a semiclassical $0$-pseudodifferential operator of order $(0, k)$ (the first index is the semiclassical order, the second the differential order) takes the form
\begin{equation}
h^{-(n+1)} \int e^{i(z - z') \cdot \zeta/h} b(z, \zeta) \, d\zeta
\label{psdo-int}\end{equation}
(where $b$ is a symbol of order $k$ in $\zeta$) near the diagonal and away from the boundary of $X^2_0$, and
\begin{equation}
A = h^{-(n + 1)} \int_{\mathbb{R}^{n + 1}} e^{i \big((x^{\prime\prime} - x) \lambda^{\prime\prime} + i(y^{\prime} - y) \cdot \mu^{\prime\prime} \big)/ (h x^{\prime\prime})} a(x^{\prime\prime}, y^{\prime\prime}, \lambda^{\prime\prime}, \mu^{\prime\prime}) \, d\lambda^{\prime\prime} d\mu^{\prime\prime}
\label{psdo-bdy}\end{equation}
(where $a$ is a symbol of order $k$ in $(\lambda'', \mu'')$) near the boundary of the diagonal in $X^2_0$; away from the diagonal, the kernel is smooth and $O(h^\infty \rho_L^\infty \rho_R^\infty)$.

We wish to show that by composing with pseudodifferential operators acting on $X$, we can localize the microlocal support of $u \in I^m(\Lambda)$.
More precisely, we shall establish

\begin{proposition}\label{prop:microsupport} Suppose that $\Lambda$ is a smooth Lagrangian submanifold in ${}^\Phi T^* X^2_0$, and let $U \in  I^{m}(\Lambda)$ and $A \in {}^0\Psi^{0, 0}(X)$. Then $AU \in I^m(\Lambda)$ and  we have
\begin{equation}\begin{gathered}
\WF(A U) \subset \pi_L^{-1} \big(\WF(A)\big) \cap \WF (U), \\
\WF(U A) \subset \pi_R^{-1} \big(\WF(A)\big) \cap \WF (U). \\
\end{gathered}\label{microsupport}\end{equation}
 Here  $\pi_L$, $\pi_R$ is the left, resp. right projection from ${}^\Phi T^* X^2_0 \rightarrow {}^0 T^* X$, that is, the composite map
$$
{}^\Phi T^* X^2_0 \to \big( {}^0 T^* X \big)^2 \to {}^0 T^* X
$$
where the first map is induced by the blow-down map $\beta : X^2_0 \to X^2$, and the second is the left, resp. right projection.
\end{proposition}

\begin{proof} The second statement in \eqref{microsupport} follows from the first by switching the left and right variables. So we only prove the first. To do this, we write down local parametrizations of $U$, and check the statement \eqref{microsupport} on each. We use local coordinates valid in Regions 1--5 as described in Section~\ref{sec:ahm}.

$\bullet$ Region 1. In this region, $U$ has a local representation
$$
U = h^{-m-k/2-(n+1)/2} \int_{\RR^k} e^{i\phi(z, z', v)/h} b(z, z', v, h) \, dv,
$$
and $A$ has a representation \eqref{psdo-int}. The composition is given by an oscillatory integral
$$
h^{-m-k/2-3(n+1)/2} \int e^{i((z-z'') \cdot \zeta + \phi(z'', z', v))/h} a(z, \zeta) b(z'', z', v, h) \, dv \, d\zeta \, dz''.
$$
We perform stationary phase in the variables $(z'', \zeta)$. We note that the Hessian in these variables is non-degenerate, as the matrix of second derivatives takes the form
$$
\begin{pmatrix}
* & \Id \\
\Id & 0 \end{pmatrix}
$$
which has nonzero determinant, irrespective of the top left entry. The stationary phase expansion then shows that this expression can be simplified to
$$
h^{-m-k/2-(n+1)/2} \int e^{i\phi(z, z', v)/h} c(z, z', v, h) \, dv  + O(h^\infty), 
$$
where $c$ has an expansion
$$
c(z, z', v, h) = \sum_{j=0}^\infty h^j Q_j \big(a(z, \zeta) b(z'', z', v, h) \big) \Big|_{z'' = z, \zeta = -d_z \phi(z, z', v)}
$$
where $Q_j$ is a differential operator in the $(z'', \zeta)$ variables of degree $2j$.
This shows that $AU \in I^m(\Lambda)$ and has microlocal support contained in $\WF(U)$ (since the amplitude $c$ is $O(h^\infty)$ wherever $b = O(h^\infty)$). The microlocal support is also contained in the set
$$
\{ (z, z', v, h)   \mid (z, d_z \phi(z, z', v)) \in  \WF(A) \},
$$
which is to say that the microlocal support is contained in $\WF(U) \cap \pi_L^{-1} \WF(A)$.

$\bullet$ Region 2a. In this region, $U$ has a local representation
$$
U = h^{-m-k/2-(n+1)/2} \int e^{i(\phi(x, y, z', v) \pm \log x)/h} b(x, y, z', v, h) \, dv,
$$
and $A$ has a representation \eqref{psdo-bdy}. The composition is given by an oscillatory integral
\begin{multline*}
h^{-m-k/2-3(n+1)/2} \int e^{i\big( (x-x'')\lambda/x + (y-y'')\cdot \mu/x + \phi(x'', y'', z', v) \pm \log x'' \big)/h} a(x, y, \lambda, \mu) \\ \times  b(x'', y'', z', v, h) \, dv \, \frac{dx'' \, dy''}{{x''}^{n+1}} \, d\lambda \, d\mu.
\end{multline*}
We change to coordinates $s''=x''/x$ and $Y'' = (y-y'')/x$. In these coordinates we have
\begin{multline*}
h^{-m-k/2-3(n+1)/2} \int e^{i\big( (1-s'')\lambda +Y''\cdot \mu + \phi(x s'', y-xY'', z', v)\pm \log(xs'') \big)/h} a(x y, \lambda, \mu) \\ \times  b(x s'', y-xY'', z', v, h) \, dv \, \frac{ds'' \, dY''}{{s''}^{n+1}} \, d\lambda \, d\mu.
\end{multline*}
We then perform stationary phase in the variables $(s, \lambda, Y, \mu)$. There is a stationary point at
\begin{equation}
s'' = 1, Y'' = 0, \lambda = x d_x \phi \pm 1, \mu = x d_y \phi.
\label{spt}\end{equation}
We check that the Hessian in these variables is non-degenerate at this critical point.  The matrix of second derivatives takes the form
$$
\begin{pmatrix}
* & \Id & O(x) & 0 \\
\Id & 0 & 0 & 0 \\
O(x) & 0 & * & \Id \\
0 & 0 & \Id & 0
\end{pmatrix}
$$
which has nonzero determinant when $x$ is small, irrespective of the starred entries.  The stationary phase expansion then shows that this expression can be simplified to
$$
h^{-m-k/2-(n+1)/2} \int e^{i\phi(x,y, z', v) \pm \log x/h} c(x,y, z', v, h) \, dv  + O(h^\infty),
$$
where $c$ has an expansion
$$
c(x,y, z', v, h) = \sum_{j=0}^\infty h^j Q_j \big( a(x, y, \lambda, \mu) b(x' s'', y-xY'', z', v, h) \big)  \Big|_{x'' = x, y'' = y,  \lambda = \pm 1 + x d_x \phi, \mu = x d_y \phi}
$$
where $Q_j$ is a differential operator in $(s'', \lambda, Y'', \mu)$ of degree $2j$.
This shows that $AU \in I^m(\Lambda)$, and has microlocal support contained in $\WF(U)$ (since $c = O(h^\infty)$ wherever $b = O(h^\infty)$). The microlocal support is also contained in 
$$
\{ (x, y, z', v, h) \mid (x, y, \pm 1 + xd_x\phi(x, y, z', v), x d_y \phi(x,  y, z', v)) \in \WF(A) \},
$$
which (comparing with \eqref{spt}) shows that the microlocal support of $AU$ is also contained in $\pi_L^{-1} \WF(A)$.

$\bullet$ Region 2b. In this region, the calculation is similar to region 1, so we omit the details.

$\bullet$ Region 3. In this region, the calculation is similar to region 2a, so again we omit the details.

$\bullet$ Region 4a. Here we use the coordinates
$$
s = \frac{x}{x'}, \ x', \ y, \ Y = \frac{y' - y}{x'}.
$$
In this region, $U$ has a local representation
$$
U = h^* \int e^{i(\phi(s, x', y, Y, v) \pm \log s)/h} b(s, x', y, Y , v, h) \, dv,
$$
and $A$ has a representation \eqref{psdo-bdy}. The composition is given by an oscillatory integral
\begin{multline*}
h^{-m-k/2-3(n+1)/2} \int e^{i\big( (x-x'')\lambda/x + (y-y'')\cdot \mu/x  + \phi(x''/x', x', y', (y'-y'')/x', v) \pm \log (x''/x')\big)/h} \\ \times a(x, y, \lambda, \mu) b(x''/x', x', y'', (y'-y'')/x',  v, h) \, dv \, \frac{dx'' \, dy''}{{x''}^{n+1}} d\lambda d\mu
\end{multline*}
We introduce coordinates $Y'' = (y-y'')/x$, $s'' = x''/x$. The integral becomes
\begin{multline*}
h^{-m-k/2-3(n+1)/2} \int e^{i\big( (1-s'')\lambda + Y'' \cdot \mu  + \phi(s''s, x', y'', sY''+Y, v) \pm \log (s''s) \big)/h} \\ \times a(x, y, \lambda, \mu) b(s'' s, x', y'', sY''+Y,  v, h) \, dv \, \frac{ds'' dY''}{{s''}^{n+1}} d\lambda d\mu .
\end{multline*}

We perform stationary phase in the variables $(s'', Y'', \lambda, \mu)$. It is straightforward to check that the Hessian in these variables is nondegenerate at the stationary point
\begin{equation}
s'' = 1, \ Y'' = 0 , \ \lambda = s d_s \phi \pm 1, \ \mu = x d_y \phi - d_Y \phi.
\label{spt2} \end{equation}
We then get a stationary phase expansion, as in the previous regions, leading to the conclusion that $AU$ has an expression
$$
U = h^{-m-k/2-(n+1)/2} \int e^{i(\phi(s, x', y, Y, v) \pm \log s)/h} c(s, x', y, Y , v, h) \, dv  + O(h^\infty),
$$
such that $c$ is given in terms of $a $ and $b$ by a stationary phase expansion as in Regions 1 or 2a above. Thus, $AU$ is a Lagrangian distribution in $I^m(\Lambda)$, and $c$ is $O(h^\infty)$ wherever $b = O(h^\infty)$, and is supported
where $(x, y, \lambda, \mu) \in \WF(A)$. It follows (using \eqref{spt2}) that $\WF(AU)$ is contained in $\WF(U) \cap \pi_L^{-1}(\WF(A))$.

$\bullet$ Region 4b. This is given by a rather similar calculation to region 4a, which we omit.

$\bullet$ Region 5. Here we use the coordinates
$$
s_1 = \frac{x}{y_1' - y_1}, \ s_2 = \frac{x'}{y_1' - y_1}, \ t = y_1' - y_1, \ y', Z_j = \frac{y_j' - y_j}{y_1' - y_1}, \quad j \geq 2.
$$
In this region, $U$ has a local representation
$$
U = h^{-m-k/2-(n+1)/2} \int e^{i(\phi(s_1, s_2, t, y', Z,  v) \pm \log(s_1 s_2))/h} b(s, x', y, Y , v, h) \, dv.
$$
Writing $s'' = x/x''$ and $Y'' = (y-y'')/x$ as before, the composition is given by an oscillatory integral
\begin{multline*}
\int \exp \bigg\{ \frac{i}{h} \Big( (1 - s'')\lambda + Y'' \cdot \mu  + \phi \big ( \frac{s_1s''t}{t-xY''_1}, \frac{s_2t}{t-xY''_1}, t-xY''_1, y', \frac{tZ_j + x Y''_j}{t-xY''_1},  v \big) \pm \log \frac{s_1 s_2 s'' t^2}{(t-xY''_1)^2} \Big) \bigg\} \\ \times h^{-m-k/2-3(n+1)/2}  a(x, y, \lambda, \mu) b \big( \frac{s_1s''t}{t-xY''_1}, \frac{s_2t}{t-xY''_1}, t-xY''_1, y', \frac{tZ_j + x Y''_j}{t-xY''_1},  v, h\big)  \, dv \, \frac{ds'' \, dY''}{{s''}^{n+1}} d\lambda d\mu
\end{multline*}
We perform stationary phase in the variables $(s'', Y'', \lambda, \mu)$. There is a critical point at
\begin{multline}
s'' = 1, \ Y'' = 0, \ \lambda = s_1 d_{s_1} \phi, \ \mu = s_1 \Big(  s_{s_1} \phi + s_2 d_{s_2} \phi + Z \cdot d_Z \phi - t d_t \phi - s_1 d_{Z_1} \phi \Big), \ \\ \mu_j = -s_1 d_{Z_j} \phi, \ j \geq 2.
\label{spt3}\end{multline}
 It is straightforward to check that the Hessian in these variables is nondegenerate. We then get a stationary phase expansion, as in the previous regions, leading to the conclusion that $AU$ has an expression
$$
U = h^{-m-k/2-(n+1)/2} \int e^{i(\phi(s_1, s_2, t, y', Z, v) \pm \log (s_1 s_2))/h} c(s_1, s_2, t, y', Z, v, h) \, dv  + O(h^\infty),
$$
such that $c$ is given in terms of $a $ and $b$ by a stationary phase expansion. Thus, $AU \in I^m(\Lambda)$, and using the same reasoning as above, its microlocal support is contained in   $\WF(U) \cap \pi_L^{-1} \WF(A)$.

\end{proof}


\section{The spectral measure at high energy}\label{sec:he}

In this section, we prove Theorem~\ref{thm:kernelbounds1} for high
energies, $\blambda \geq 1$, which immediately implies also
Theorem~\ref{thm:kernelbounds}. Our first task  is to choose an
appropriate partition of the identity operator. This is done in
exactly the same way as was done in \cite{Guillarmou-Hassell-Sikora}
in the asymptotically conic case.

Before getting into the details we explain the advantage of using a  partition of the identity. It is to microlocalize the spectral measure (taking advantage of the microlocal support estimate, Proposition~\ref{prop:microsupport}) so that only the Lagrangian $\Lambda^{nd}$ is relevant, while the other part, $\Lambda^*$, disappears. This is important in our pointwise estimate, as the Lagrangian $\Lambda^{nd}$ locally projects diffeomorphically to the base manifold except at where it meets $N^* \diagz$, i.e. the projection ${}^\Phi \pi$, restricted to $\Lambda^{nd}$ has maximal rank except at the intersection with $N^* \diagz$, which leads to the most favourable $L^\infty$ estimates. (The drop in rank at the diagonal leads to the different form of the estimates for small $d(z,z')$ in Theorem~\ref{thm:kernelbounds}.) By contrast, we cannot control the rank of the projection from $\Lambda^*$ to the base (except by making additional geometric assumptions, such as nonpositive curvature of $X^\circ$, which we do in Sections~\ref{sec:restriction} and  \ref{sec:sm}).

\subsection{Partition of the identity}
Our operators $Q_i(\blambda)$ will be semiclassical
$0$-pseudodifferential operators of order $(0,0)$, where the first
index denotes the semiclassical order and the second, the
differential order, and the semiclassical parameter is $h = \blambda^{-1}$. In fact, all but $Q_0(\blambda)$ will have differential order $-\infty$.

First of all, we will choose $Q_{0}$ of order $(0,0)$ microlocally
supported away from the characteristic variety of $h^2 \Delta - 1$,
say in the region $\{\sigma(h^2 \Delta) \in [0, 3/4] \cup [5/4,
\infty) \}$, and microlocally equal to the identity in a smaller
region, say $\{\sigma(h^2 \Delta) \in [0, 1/2] \cup [3/2, \infty)
\}$.  In light of the disjointness of semiclassical wavefront sets,
the term $Q_{0}(\blambda) \spmeas Q_0(\blambda)$ has empty
microlocal support, and is therefore $O(h^\infty)$. Taking into
account the behaviour at the boundary, we find that
\begin{equation}
Q_{0}(\blambda) \spmeas Q_0(\blambda) \in h^\infty (\rho_L
\rho_R)^{n/2 - i/h} C^\infty(X^2_0 \times [0, h_0]) + h^\infty
(\rho_L \rho_R)^{n/2 + i/h} C^\infty(X^2_0 \times [0, h_0]).
\label{Q0term}\end{equation} This clearly satisfies
Theorem~\ref{thm:kernelbounds1}.

We next choose a cutoff function $\chi(x)$, equal to $1$ for $x \leq
\epsilon$ and $0$ for $x \geq 2\epsilon$. We decompose the remainder
$\Id - Q_0(\blambda)$ into $(\Id - Q_0(\blambda)) \chi(x)$ and $(\Id
- Q_0(\blambda)) (1-\chi(x))$, and further decompose these two
pieces in the following way.

We divide the interval $[-3/2, 3/2]$ into a union of intervals $B_i$
with overlapping interiors, and with diameter $\leq \beta$. We then
decompose $(\Id - Q_0(\blambda)) \chi(x)$ into operators
$Q_i(\blambda), \dots, Q_{N_1}(\blambda)$ such that each operator
$Q_i(\blambda)$ has wavefront set contained in $\{ \lambda (\lambda^2 + h^{ij} \mu_i \mu_j)^{-1/2}  \subset
B_i \}$.

Next, we decompose $(\Id - Q_0(\blambda)) (1-\chi(x))$.  The idea is
still to decompose this operator into pieces, so that on each piece
the microlocal support is small. Let $d(\cdot, \cdot)$ be the Sasaki
distance on $T^* X^\circ$. We break up $(\Id - Q_0(\blambda))
(1-\chi(x))$ into a finite number of operators $Q_{N_1 +
1}(\blambda), \dots, Q_{N_1 + N_2}(\blambda)$, each of which is such
that the microlocal support has diameter $\leq \eta$ with respect to
the Sasaki distance on $T^* X^\circ$. This is possible since the
microlocal support of $(\Id - Q_0(\blambda)) (1-\chi(x))$ is compact
in $T^* X^\circ$. We choose $\eta < \iota/4$, where $\iota$ is the
injectivity radius of $(X^\circ, g)$.

We now prove a key property about the microlocal support of
$Q_i(\bsigma) \spmeas Q_i(\bsigma)^*$, when $\epsilon$, $\beta$
and $\eta$ are sufficiently small.

\begin{proposition}\label{prop:Qi}
Suppose that $\epsilon$, $\beta$ and $\eta$ are sufficiently small.
Then, provided that $\bsigma$ and $\blambda$ satisfy $\bsigma \in [(1-\delta) \blambda, (1+\delta) \blambda]$ for sufficiently small $\delta$,  the microlocal support of $Q_i(\bsigma) \spmeas
Q_i(\bsigma)^*$, $i \geq 1$, is a subset of $\Lambda^{nd}$.
\end{proposition}

\begin{remark} In the composition $Q_i(\bsigma) \spmeas Q_i(\bsigma)^*$, we view all operators as semiclassical operators with parameter $h = \blambda^{-1}$. To do this for $Q_i(\bsigma)$, we need to scale the fibre variables in the symbol by a factor of $\blambda/\bsigma$. This is of little consequence as   $\blambda/\bsigma$ is close to $1$ by assumption. 
\end{remark}

\begin{proof}
Recall that $\Lambda^{nd}$ consists of a neighbourhood $U_1$ of $\partial_{\FF}\Lambda$ in $\Lambda$, together with a neighbourhood $U_2$ of $\Lambda \cap N^* \diagz$ in $\Lambda$.

First suppose that $i=0$. By Proposition~\ref{prop:microsupport},
the microlocal support of $Q_0(\bsigma) \spmeas Q_0(\bsigma)^*$ is
empty for sufficiently small $\delta$, so the conclusion of Proposition~\ref{prop:Qi} trivially
holds.

Next suppose that $1 \leq i \leq N_1$. We claim that if $\epsilon$, $\beta$ and $\delta$ are sufficiently small, then the microsupport of
$Q_i(\bsigma) \spmeas Q_i(\bsigma)^*$ is contained in $U_1$. 

Let
$U_1' = U_1 \setminus  \partial_{\FF}\Lambda$, i.e. a deleted
neighbourhood of $\partial_{\FF}\Lambda$. Since the microlocal
support is always a closed set it suffices to show that
$$
\WF \big(Q_i(\bsigma) \spmeas Q_i(\bsigma)^* \big) \setminus \{
\rho_F = 0 \} \text{  is contained in } U_1'.
$$

 By Proposition~\ref{prop:microsupport}, this wavefront set is contained in
\begin{equation*}\begin{gathered}
\big\{ (z, \zeta; z', -\zeta') \mid |\zeta|_g = |\zeta'|_g = 1, \ (z,  \zeta) \text{ and } (z', \zeta') \text{ lie on the same geodesic, } \\
 (z, \frac{\blambda}{\bsigma}\zeta), (z', \frac{\blambda}{\bsigma}\zeta')  \in \WF(Q_i(\bsigma)) \big\}.
\end{gathered}\end{equation*}

We prove the claim by contradiction. Suppose that the claim were false. Choose sequences $\beta_k$, $\delta_k$ and $\epsilon_k$ tending to zero as $k \to \infty$, and for each $k$, a partition of the identity $Q_i^{(k)}$ satisfying the conditions above relative to $\beta_k$, $\delta_k$ and $\epsilon_k$. Then, if the claim is false for all $Q_i^{(k)}$, there are sequences of pairs of points $(x_k, y_k, \lambda_k, \mu_k)$, $(x'_k, y'_k, \lambda'_k, \mu'_k)$ in $\WF(Q^{(k)}_i)$, lying on the same geodesic $\gamma_k$,  with $x_k, x'_k \leq \epsilon_k$, $|\lambda_k - \lambda'_k| \leq \beta_k(1-\delta_k)^{-1}$, but the corresponding point $(x_k, y_k, \lambda_k, \mu_k; x'_k, y'_k, -\lambda'_k, -\mu'_k)$ not in $U_1'$. By compactness we can take a convergent subsequence, with $x_k \to 0$, $x'_k \to 0$, $y_k \to y_0$, $y'_k \to y'_0$, $\lambda_k, \lambda_k' \to \lambda_0$,
$\mu_k \to \mu_0$, $\mu'_k \to \mu'_0$.
Consider the limiting behaviour of the geodesic $\gamma_k$ connecting $(x_k, y_k, \lambda_k, \mu_k)$ and $( x'_k, y'_k, -\lambda'_k, -\mu'_k)$. If $\lambda_0 \neq \pm 1$ then $\gamma_k$ converges to a boundary bicharacteristic, that is, an integral curve of \eqref{ge} contained in $\{ x = 0 \}$, and therefore takes the form
$$
x(\tau) = 0, \ y(\tau) = y^*, \ \lambda(\tau) = \cos \tau, \ \mu(\tau) = \sin \tau \mu^*, \text{ where } \frac{d\tau}{dt} = \sin \tau.
$$
(It is straightforward to check that this satisfies the geodesic equations \eqref{ge} in the parameter $t$.)
Therefore, $(y_0, \lambda_0, \mu_0) = (y^*, \cos \tau, \sin \tau \mu^*)$ and $(y'_0, \lambda'_0, \mu'_0) = (y^*, \cos \tau', \sin \tau' \mu^*)$ for some $\tau$ and $\tau'$. Since $|\lambda_k - \lambda'_k| \to 0$, we  have $\tau=\tau'$, and it then follows that $\mu_0 = \mu'_0$. This shows that the limiting point lies on $\partial_{\FF} \Lambda$, in fact over the fibre $F_{y^*}$ of $\FF$ lying over $y^*$ (over which point on this fibre depends on the limiting values of $x/x'$ and $(y'-y)/x'$).  Hence the sequence converging to it eventually lies in $U_1'$, which is our desired contradiction.

If $\lambda_0 = \pm 1$ then the limiting geodesic could be an interior bicharacteristic. In this case we must have $\lambda_0, \lambda_0' \in \{ \pm 1 \}$, i.e the points $(x_0, y_0, \lambda_0, \mu_0)$ and $( x_0', y_0', -\lambda_0', \mu_0')$ are both an endpoint of this bicharacteristic. However the condition that the difference $|\lambda - \lambda'|  \to 0$ along the sequence means that either both $\lambda_0, \lambda_0'$ are $+1$ or both are $-1$. Thus, the limiting points $(x_0, y_0, \lambda_0, \mu_0)$ and $( x_0', y_0', -\lambda_0', \mu_0')$ are again equal in this case, 
with $x_0 = x_0' = 0$, $\mu_0 = \mu_0' = 0$, which shows that the limiting point lies on $\partial_{\FF} \Lambda$, hence the sequence converging to it eventually lies in $U_1'$, again producing a  contradiction.

We next claim that if $N_1 + 1 \leq i \leq N_1 + N_2$, then for
$\eta$ sufficiently small, the wavefront set of $Q_i(\bsigma)
\spmeas Q_i(\bsigma)^*$ is contained in $U_2$. The argument is
similar. Choose a sequence $\eta_k$ tending to zero as $k \to
\infty$, and for each $k$, a partition of the identity $Q_i^{(k)}$
satisfying the conditions above relative to  $\eta_k$. Then, if the
claim is false for all $Q_i^{(k)}$,  then we could find  a sequence
 $(z_k, \zeta_k), (z_k', \zeta_k')$ on the same geodesic, with  $(z_k,
\zeta_k), (z_k', \zeta_k')\in \WF(Q_i)$, with each $(z_k, \zeta_k; z_k',
-\zeta_k') $ not in $U_2$. Using compactness we can extract a
convergent subsequence from the $(z_k, \zeta_k)$, converging to $(z_0,
\zeta_0)$. Since $\eta_k \to 0$ the sequence $(z_k', \zeta_k')$ also converges to $(z_0, \zeta_0)$. But the point $(z_0, \zeta_0, z_0,
-\zeta_0)$ is in $N^* \diagz$, and $U_2$ is a neighbourhood of $N^*
\diagz$ in $\Lambda$, so this gives us the contradiction.

\end{proof}

We now assume that $\epsilon, \delta, \eta$ have been chosen small enough that the conclusion of Proposition~\ref{prop:Qi} is valid.

\subsection{Pointwise estimates for microlocalized spectral measure near the diagonal}

In this section we show that an element $U$ of $I^{-1/2}(X^2_0, \Lambda^{nd}; \Omegazh)$ satisfies \eqref{bpm} and \eqref{bpm2} in the region $d(z,z') \leq 1$.

We divide this into the case where we work away from the boundary of $\diagz$, and near a point on the boundary of the diagonal. The first case, localizing away from the boundary of the diagonal, has been treated in \cite[Proposition 1.3]{Hassell-Zhang} (this was done in the context of asymptotically conic manifolds, but away from the boundary, one `cannot tell' whether one is on an asymptotically conic or asymptotically hyperbolic manifold, so the argument applies directly).\footnote{Notice that the dimension was denoted $n$ in \cite{Hassell-Zhang}, instead of $n+1$, as here, when comparing \cite[Equation (1.13)]{Hassell-Zhang} with \eqref{bpm}.}

Thus, it remains to deal with the case where $\spmeas$ is microlocalized to a neighbourhood of a point $q \in \partial_{\FF} N^* \diagz \cap \Lambda^{nd}$. In this case, any parametrization of the Lagrangian $\Lambda^{nd}$ must have at least $n$ integrated variables, since the rank of the projection from $\Lambda^{nd}$ to the base $X^2_0$ drops by $n$ at $N^* \diagz$.

The following result is essentially taken from \cite{Hassell-Zhang}.

\begin{proposition}\label{prop:param}
Let $q$ be a point in $\partial_{\FF} N^* \diagz \cap \Lambda^{nd}_\pm$, and let $U \in I^{-1/2}(X^2_0, \Lambda^{nd}; \Omegazh)$.  Then microlocally near $q$, $U$ can be represented as an oscillatory integral of the form
\begin{equation}
h^{-n} \int_{\RR^n} e^{i\Psi(x', y, r,w, v)/h} a(x', y, r, w, v, h) \, dv,
\label{oscint-nd}\end{equation}
where the coordinates $(r, w)$, $w = (w_1, \dots, w_{n})$ define the boundary, that is, $\diagz = \{ r = 0, w = 0 \}$ and the differentials $dr$ and $dw_i$ are linearly independent. Here also, $v = (v_1, \dots, v_{n}) \in \RR^n$,  and $a$ is smooth and compactly supported in all variables. Moreover, we may assume that $\Psi$ has the properties
\begin{equation}
\begin{cases}
\text{(a) } d_{v_j} \Psi = w_j + O(r), \\
\text{(b) }\Psi = \sum_{j=1}^{n} v_j \partial_{v_j} \Psi + O(r), \\
\text{(c) }d^2_{v_j v_k} \Psi = r A_{jk}, \text{ where $A$ is a smooth, nondegenerate matrix} \\
\text{(d) }d_v \Psi = 0 \implies \Psi = \pm d(z, z').
\end{cases}
\label{Psi-properties}\end{equation}
\end{proposition}

\begin{proof} In this region, $(x', y, s=x/x', Y = (y'-y)/x')$ furnish local  coordinates on $X^2_0$. In these coordinates, the diagonal is defined by $s = 1, Y = 0$. We will write $(r, w)$ for a suitable rotation of the coordinates $(s-1, Y)$. Let $(\xi', \eta, \rho, \kappa)$ be the dual coordinates to $(x', y, r, w)$. We claim that, for some  rotation $(r, w)$ of the $(s-1, Y)$ coordinates, we have $d\rho |_{N^* \diagz \cap \Lambda^{nd}}= 0$ at $q$. This follows from the fact that $N^* \diagz \cap \Lambda^{nd}$ projects to $\diagz$ with $n$-dimensional fibres (in fact $N^* \diagz \cap \Lambda^{nd}$ is an $S^n$-bundle over $\diagz$), so the $n+1$ differentials $d\rho, d\kappa_1, \dots, d\kappa_n$ span an $n$-dimensional space in $T_q (N^* \diagz \cap \Lambda^{nd})$ for each $q \in N^* \diagz \cap \Lambda^{nd}$.

Secondly, we claim $d r|_{\Lambda} \neq 0$ at $q$. To see this, we observe that there must be a vector $V\in T_q \Lambda$ tangent to $\Lambda$ but not tangent to $N^* \diagz$. Therefore, $V$ must have a non-zero $\partial_{w_j}$ or $\partial_r$ component. Also, the vectors $\partial_{\kappa_i}$ are tangent to $\Lambda$ at $q$, since $T_q (N^* \diagz \cap \Lambda^{nd})$ is the codimension 1 subspace of $T_q (N^* \diagz)$ given by the vectors annihilated by $d\rho$. Thus, as $\Lambda$ is Lagrangian, we have
$$
\omega(V, \partial_{\kappa_i}) = \Big(  dy \wedge d {\eta} + dr \wedge d\rho + dw \wedge d\kappa + dx^\prime \wedge d{\xi}^\prime \Big) (V, \partial_{\kappa_i}) = 0.
$$
This implies that $dw_j(V) = 0$ for each $j$, which implies $V$ has a non-zero $\partial_r$ component, as claimed.

This shows that $(x', y, r, \kappa)$ furnish coordinates on $\Lambda$ locally near $q$. Thus, we can express the  remaining coordinates, restricted to $\Lambda$, as smooth functions of these:
\begin{gather}
w = W(x^\prime, y, r,\kappa), \quad
\rho = R(x^\prime, y, r,\kappa), \quad
{\xi}' = {\Xi}'(x^\prime, y, r,\kappa), \quad
{\eta} = {H}(x^\prime, y, r,\kappa).
\label{express}\end{gather}
Also, using the fact that $\Lambda$ is Lagrangian, the form
\begin{equation}
d \Big( \xi' dx' + \eta \cdot dy + \rho dr + \sum_j \kappa_j  dw_j \Big) = 0 \text{ on } \Lambda.
\label{omega-vanishes}\end{equation}
It follows that there is a function $f(x', y, r, \kappa)$ on $\Lambda$, defined near $q$,  such that
\begin{equation}
 \Xi' dx' + H \cdot dy + R dr + \sum_j \kappa_j dW_j = df.
\label{df}\end{equation}
Notice that $\Lambda \cap \{ r = 0 \} = \Lambda \cap N^* \diagz$, and at $N^* \diagz$ we have $\xi' = 0, \eta = 0, w = 0$. Therefore, at $r = 0$, we have
$$
\frac{\partial f}{\partial x'} = 0, \quad \frac{\partial f}{\partial y} = 0, \quad \frac{\partial f}{\partial \kappa} = 0.
$$
It follow that $f$ is constant when $r=0$. Since $f$ is undetermined up to a constant, we may assume that $f=0$ when $r=0$; that is, $f(x', y, r, \kappa) = r \tilde f(x', y, r, \kappa)$.

We claim that the function
$$\Psi = \sum_{j=1}^n \big(w_i - W_i(x^\prime, y, r, v)\big) v_i + f(x^\prime, y, r, v)$$
locally parametrizes the Lagrangian $\Lambda$ near $q$, and satisfies properties (a) -- (d) above.

To check that $\Psi$ parametrizes $\Lambda$, we set $d_v \Psi = 0$. This implies that
\begin{equation}
w_i - W_i(x^\prime, y, r, v) = \sum_i v_j \frac{\partial W_j}{\partial v_i} - \frac{\partial f}{\partial  v_i}.
\label{dvPsi}\end{equation}
On the other hand, the 1-form identity \eqref{df} shows that the
functions $W, R, \Xi', H$ and $f$ satisfy the identities
\begin{equation}\begin{aligned}
\Xi'(x', y, r, v) = -\sum_j v_j \frac{\partial W_j}{\partial x'}(x', y, r, v) + \frac{\partial f}{\partial x'}(x', y, r, v) &= \frac{\partial \Psi}{\partial x'} \\
H_k = -\sum_j v_j \frac{\partial W_j}{\partial y_k} + \frac{\partial f}{\partial y_k} &= \frac{\partial \Psi}{\partial y_k} , \\
R = - \sum_j v_j \frac{\partial W_j}{\partial r} + \frac{\partial f}{\partial r} &= \frac{\partial \Psi}{\partial r} , \\
\sum_j v_j \frac{\partial W_j}{\partial v_i} &= \frac{\partial f}{\partial v_i}.
\end{aligned}\end{equation}
The last of these identities shows that the RHS of \eqref{dvPsi} vanishes. We therefore find that the Lagrangian parametrized by $\Psi$ is
\begin{equation}\begin{gathered}
\Big\{  (x', y, r, w; \xi', \eta, \rho, \kappa) \mid \xi' = d_{x'} \Psi, \  \eta = d_{y} \Psi,  \  \rho = d_r \Psi,  \  \kappa = d_w \Psi \Big\} \\
= \Big\{  (x', y, r, w; \xi', \eta, \rho, \kappa) \mid \xi' = \Xi' , \   \eta = H',  \  \rho = R, \   w = W \Big\} \\
= \Lambda.
\end{gathered}\end{equation}

It follows that, microlocally near $q$, the spectral measure may be written as an oscillatory integral with phase function $\Psi$, as in \eqref{oscint-nd}, where the power of $h$ is given by $-m - N/4 - k/2 = -n$ where $m = -1/2$ is the order of the Lagrangian distribution, $N = 2(n+1)$ is the spatial dimension and $k=n$ is the number of integrated variables.

Conditions (a) and (b) are easily verified, using the fact that $W$ and $f$ are $O(r)$. To check condition (c), we write
$$d_{vv}^2 \Psi = r A + O(r^2),$$
where $A$ is an $n \times n$ matrix function of $(r, x^\prime, y, v)$. We claim that $A$ is invertible at $q$. It suffices to check $d_{vv}^2 \Psi = O(r^n)$ near $q$. On one hand, since $\Psi$ is a phase function parametrizing $\Lambda$ nondegenerately in a neighbourhood of $q$, then we have a local diffeomorphism $$\{(r, x', y, v)\}  \longrightarrow \{(r, x' ,y, \Psi^\prime_r, \Psi^\prime_{x'}, \Psi_y^\prime, \Psi'_v)\}.$$ The determinant of the differential of the map \begin{equation}\label{eqn: projection map variant}\{ (r, x', y, \Psi^\prime_r, \Psi^\prime_{x'}, \Psi_y^\prime, \Psi'_v) \} \longrightarrow \{ (r, x', y, \Psi'_v) \}\end{equation} is thus equal to the determinant of the differential of the map $$\{(r, x', y, v)\} \longrightarrow \{ (r, x', y, \Psi'_v) \},$$ which is simply $\det d^2_{vv} \Psi.$ On the other hand, it is obvious that the determinant of the differential of map (\ref{eqn: projection map variant}) equals the determinant of the projection map \begin{equation*}\pi : \{ (r, x', y, \Psi^\prime_r, \Psi^\prime_{x'}, \Psi_y^\prime) : \Psi'_v = 0 \} \longrightarrow \{ (r, x', y) \}\end{equation*}
from $\Lambda$ to $X^2_0$.
Since $\det d\pi = O(r^n)$ near $q$ \cite[Proposition 15]{Chen-Hassell1}, we get $$\det d^2_{vv} \Psi = O(r^n) \quad \mbox{near $q$},$$ which implies $A$ is invertible.

To check (d), we notice that if $r=0$ and $d_v \Psi= 0$, then $\Psi = 0$. Also, $\Lambda \cap \{ r = 0 \} = \Lambda \cap N^* \diagz$, so this is at the diagonal, i.e. $d(z, z') = 0$, so this agrees with (d). On the other hand, if $r \neq 0$, then condition (c) says that $d^2_{vv} \Psi$ is nondegenerate. In this case we can eliminate the extra $v$ variables, and reduce to a function of $(x', y, r, w)$ that parametrizes the Lagrangian $\Lambda$ locally. (This is an analytic reflection of the geometric fact that the projection $\pi : {}^\Phi T^* X^2_0 \to X^2_0$ has full rank restricted to $\Lambda$, for $r \neq 0$.)  But then \cite[Proposition 16]{Chen-Hassell1} implies that
the value of $\Psi$, when $d_v \Psi = 0$, is equal to $d(z, z') + c$ on
the forward half of $\Lambda$ (with respect to geodesic flow), and $-d(z, z') + c'$ on the backward half of $\Lambda$. Condition (b) implies that $c=c' =0$. This completes the proof of Proposition~\ref{prop:param}
\end{proof}

Using this we now show

\begin{proposition}\label{prop:d<1} Suppose $U \in I^{-1/2}(X^2_0, \Lambda^{nd}; \Omegazh)$. Then, $U$ can be written in the form \eqref{bpm} with the amplitude functions $b_\pm$ satisfying \eqref{bpm2} in the region $d(z,z') \leq 1$.
\end{proposition}

\begin{proof} We estimate the integral \eqref{oscint-nd} by dividing into three cases, depending on the relative size of $r$, $|w|$ and $h$.

\textbf{Case 1.} $|r| \leq h$. Since $r = 0$ at $\Lambda^{nd} \cap N^* \diagz$, and $dr \neq 0$ there, $|r|$ is comparable to $d(z, z')$. So in this case, we have $d(z,z') = O(h)$. Thus, we need to show  (comparing to \eqref{oscint-nd})
$$
\Big( h \frac{d}{dh} \Big)^j \bigg( e^{\mp i d(z,z')/h} 
\int_{\RR^n} e^{i\Psi(x', y, r,w, v)/h} a(x', y, r, w, v, h) \, dv \bigg) = O(1).
$$
For $j=0$ this is trivial. Consider $j=1$. We claim that
this is also $O(1)$. This differential operator is certainly
harmless when applied to the amplitude, $a$. When applied to the
exponential, it brings down a factor $i (\pm d(z,z') +  \Psi)/h$,
which using (b) and $r, d(z,z') = O(h)$, we write $i h^{-1} v_i
\Psi_{v_i} + O(1)$. 
The  term $i h^{-1}  v_i \Psi_{v_i}$ times the exponential $e^{i\Psi/h}$ is
equal to $v_i d_{v_i} e^{i\Psi/h}$. We integrate by parts, shifting
the $v_i$ derivative to the amplitude $a$. In this way we see that
the result of applying $h \partial_h$ to the expression is still
$O(1)$. A similar argument applies to repeated applications of
$h \partial_h$. Thus this term takes the form \eqref{bpm}.

\textbf{Case 2.} $|r| \leq c|w|$ for some small constant $c$.

In this case, there must be some $w_j$ such that $|r| \leq c |w_j|$. Then $d_{v_j} \phi = w_j + O(r) \neq 0$ in a neighbourhood of $q$, provided $c$ is sufficiently small. We can then integrate by parts arbitrarily many times in $v_j$, obtaining infinite order vanishing in $h$. The same is true for any number of $h \partial_h$ derivatives applied to \eqref{oscint-nd}.
Thus this term satisfies \eqref{apm}.

\textbf{Case 3.} $|r| \geq h$ and $|r| \geq c |w|$, with $c$ as in Case 2. 

The idea for this region is to use a stationary phase estimate, as in \cite[Section 4]{Hassell-Zhang}. We follow this proof almost verbatim; the changes required here are mostly notational.  In this case, we show a representation of the form \eqref{bpm}, \eqref{bpm2}. Notice that if $d_v \Psi = 0$ and $|r| \geq h$ then locally there are two sheets $\Lambda_\pm$ of $\Lambda$ above $X^2_0$ --- see Proposition~\ref{prop:flowout-structure}. On $\Lambda_\pm$  we divide by $e^{\pm id(z,z')/h}$ and show an estimate of the form \eqref{bpm2}. The argument for each is the same, so we only describe the argument for $\Lambda_+$. Thus, we define, with $d = d(z, z')$,
\begin{equation}
b(x', y, r, w, h) = e^{-id/h} h^{-n} \int e^{i\Psi(x', y, r, w, v)/h} a(x', y, r, w, v, h) \, dv,
\label{ba}\end{equation}
and seek to prove the estimate
\begin{equation}
\Big| (h \partial_h)^\alpha b \Big| \leq C \big(1 + \frac{|r|}{h} \big)^{-n/2},
\label{hb}\end{equation}
since in case 3, we have $|r| \sim |(r, w)| \sim d$.

Define
\begin{equation}\label{3.27}
\begin{split}
\widetilde{\Psi}(x',y,r,w,v)=r^{-1}\big(\Psi(x',y,r,w,v) - d(z,z')\big),
\end{split}
\end{equation}
and let ${\tilde \blambda}=r/h$. Notice that this function
$\widetilde \Psi$ is  $C^\infty$ in $v$, and \eqref{Psi-properties} implies that all $v$-derivatives (of all orders) are uniformly bounded in the region $|r| \geq h$ and $|r| \geq c |w|$. Then the LHS of \eqref{hb} is
\begin{equation*}
\begin{split}
h^\alpha \partial_h^\alpha b(x', y, r, w, v,
h)&=\sum_{\beta+\gamma=\alpha}\frac{\alpha!}{\beta!\gamma!}{\tilde
\blambda}^\beta\int e^{ i{\tilde \blambda}
\widetilde{\Psi}(x',y,r,w,v)}\widetilde{\Psi}^{\beta} \big(h^\gamma
\partial_h^{\gamma} a\big)(x', y, r, w, v, h) \, dv.
\end{split}
\end{equation*}
Therefore, we need to show that, for any $\beta$, we have
\begin{equation}\label{3.28}
\Big|\int e^{ i{\tilde \blambda}
\widetilde{\Psi}(x',y,r,w,v)}({\tilde
\blambda}\widetilde{\Psi})^{\beta} a(x', y, r, w, v, h) \, dv
\Big|\leq C {\tilde \blambda}^{-\frac{n}2}.
\end{equation}

We now fix $(x', y, r, w)$ with $r \geq h$.  We use a cutoff
function $\Upsilon$ to divide the $v$ integral into two parts: one
on the support of $\Upsilon$, in which $|d_v \widetilde\Psi| \geq
\tilde\epsilon/2,$ and the other on the support of $1 - \Upsilon$,
in which $|d_v \widetilde\Psi| \leq \tilde\epsilon$. On the support
of $\Upsilon$, we integrate by parts in $v$ and gain any power of
${\tilde \blambda}^{-1}$, proving \eqref{3.28}. On the support of $1
- \Upsilon$, we make a change of  variable to $\theta$ coordinates:
\begin{equation}(v_1\cdots,v_n)\rightarrow (\theta_1,\cdots,\theta_{n}),\quad \theta_i=d_{v_i}\widetilde{\Psi},~i=1\cdots, n.
\label{cov}\end{equation}
By property (c) of Proposition~\ref{prop:param},
$$
\frac{\partial \theta_j}{\partial v_k} = d^2_{v_jv_k} \widetilde \Psi = \pm A_{jk},
$$
where $A_{jk}$ is nondegenerate. This shows that this change of variables is locally nonsingular, provided $\tilde\epsilon$ is sufficiently small. Thus, for each point $v$ in the support of $1 - \Upsilon$, there is a neighbourhood in which we can make this change of variables. Using the compactness of the support of $a$ in \eqref{ba},  there are a finite number of neighbourhoods covering $\supp \Upsilon$ and the $v$-support of $a$.

On each such neighbourhood $U$, we define
$\mathcal{B}_\delta:=\big\{\theta:|\theta|\leq \delta\big\}$. Choose a  $C^\infty$ function $\chi_{\mathcal{B}_\delta}(\theta)$ which is equal to
$1$ when on the set $\mathcal{B}_\delta$ and $0$ outside
$\mathcal{B}_{2\delta}$, and with derivatives bounded by
$$
\big| \nabla^{(j)}_\theta \chi_{\mathcal{B}_\delta}(\theta) \big| \leq C \delta^{-j}.
$$
Here $\delta$ is a parameter which we will eventually choose to be
${\tilde \blambda}^{-1/2}$; however, for now we leave its value
free. Consider the integral \eqref{3.28} after changing variables
and with the cutoff function $\chi_{\mathcal{B}_\delta}(\theta)$
inserted (where we stipulate $\delta \leq \tilde\epsilon/2$, which
means that  that $1 - \Upsilon = 1$ on $\supp
\chi_{\mathcal{B}_\delta}(\theta)$):
\begin{equation*}
\Big|\int e^{i{\tilde \blambda} \widetilde{\Psi}(x', y, r, w,
\theta)}\big({\tilde \blambda}\widetilde{\Psi}\big)^{\beta} a(x', y,
r, w, \theta,
h)\chi_{\mathcal{B}_\delta}(\theta)\frac{d\theta}{|A^{-1}(x', y, r,
w,\theta)|}\Big|.
\end{equation*}
Using property (d) of Proposition~\ref{prop:param}, we see that
$\widetilde \Psi = 0$  when $\theta = 0$.  Also, from \eqref{cov}, we have  $d_\theta \widetilde \Psi = 0$ when $\theta =
0$. Hence $\widetilde \Psi = O(|\theta|^2)$. Hence
\begin{equation}
\Big|{\tilde \blambda}^\beta\int e^{ i{\tilde \blambda}
\widetilde{\Psi}(x', y, r,w,\theta)}\widetilde{\Psi}^{\beta} a(x',
y, r, w,\theta, h)\chi_{\mathcal{B}_\delta}(\theta)
\frac{d\theta}{|A^{-1}(x', y, w,\theta)|}\Big|\leq C({\tilde
\blambda}\delta^2)^\beta\delta^{n}. \label{chi}\end{equation}

It remains to treat the integral with the cutoff
$(1-\chi_{\mathcal{B}_\delta}({\theta}))$. Notice that
$|d_\theta \widetilde \Psi|$ is comparable to $|\theta|$ since
$d_\theta \widetilde \Psi = 0$ when $\theta = 0$, and
$$
d^2_{\theta_i \theta_j} \widetilde \Psi = \sum_{k,l} (A^{-1})_{il}(A^{-1})_{jk} d^2_{v_kv_l}\widetilde\Psi
$$
is nondegenerate when $\theta = 0$.
We define the
 differential operator $L$ by
\begin{equation*}
\begin{split}
L=\frac{-id_{\theta}\widetilde{\Psi}\cdot\partial_{\theta}}{{\tilde
\blambda} \big|d_{\theta}\widetilde{\Psi}\big|^2}.
\end{split}
\end{equation*}
Then the adjoint operator ${}^t L$ is given by
\begin{equation*}
{^{t}}L=-L+\frac{i}{{\tilde \blambda}}\Big(\frac{\Delta_{\theta}
\widetilde{\Psi}}{|d_{\theta}\widetilde{\Psi}|^2}-2\frac{d^2_{{\theta}_j{\theta}_k}
\widetilde{\Psi} \, d_{{\theta}_j}\widetilde{\Psi} \,
d_{{\theta}_k}\widetilde{\Psi}}{|d_{\theta}\widetilde{\Psi}|^4}\Big).
\end{equation*}
We have chosen $L$ such that $L e^{i{\tilde \blambda} \widetilde
\Psi} = e^{i{\tilde \blambda} \widetilde \Psi}$. So we introduce $N$
factors of $L$ applied to the exponential $e^{i{\tilde \blambda}
\widetilde \Psi}$ and  integrate by parts $N$ times to obtain
\begin{multline*}
\Big|\int e^{ i {\tilde \blambda}\widetilde{\Psi}(x', y, r, w,
{\theta})}({\tilde \blambda}\widetilde{\Psi})^{\beta}
a(x', y, r, w, {\theta}, h)(1-\chi_{\mathcal{B}_\delta}({\theta}))(1 - \Upsilon) \, d{\theta}\Big| \\
\leq C \int \Big| ({^{t}}L)^N\big(({\tilde \blambda}
\widetilde{\Psi})^{\beta} a(x', y, r, w, {\theta},
h)(1-\chi_{\mathcal{B}_\delta}({\theta}))(1 - \Upsilon)\big)
\Big|d{\theta}.
\end{multline*}
Inductively we find, using $|d_\theta \widetilde \Psi| \sim |\theta|$,  that
$$\big|({^{t}}L)^N\big(({\tilde \blambda}\widetilde{\Psi})^{\beta}
a(1-\chi_{\mathcal{B_\delta}})(1 - \Upsilon)\big)\big|\leq C {\tilde
\blambda}^{-N+\beta}\max\big\{|\theta|^{2\beta-2N},
|\theta|^{2\beta-N}\delta^{-N}\big\}.$$ Choosing $N$ large enough,
we get
\begin{equation}
\begin{split}
\Big|\int e^{ i {\tilde \blambda} \widetilde{\Psi}(x', y, r,
w,{\theta})}&({\tilde \blambda}\widetilde{\Psi})^{\beta} a(x', y, r,
w, {\theta}, h) (1-\chi_{\mathcal{B_\delta}})(1 - \Upsilon) \,
d{\theta}\Big|\\&\leq {\tilde
\blambda}^{-N+\beta}\int_{|\theta|\geq\delta}\big(|\theta|^{2\beta-2N}+
|\theta|^{2\beta-N}\delta^{-N}\big)d\theta\leq C {\tilde
\blambda}^{-N+\beta}\delta^{2\beta-2N}\delta^{n}.
\end{split}
\label{1-chi}\end{equation} We choose $\delta={\tilde
\blambda}^{-1/2}$ to balance the two estimates \eqref{chi} and
\eqref{1-chi}. We finally obtain
\begin{equation*}
\begin{split}
\Big|\int e^{ i {\tilde \blambda}\widetilde{\Psi}(x', y, r, w,
{\theta})}&({\tilde \blambda}\widetilde{\Psi})^{\beta} a(x', y, r,
w, {\theta}, h) (1 - \Upsilon) \,  d{\theta}\Big|\leq C {\tilde
\blambda}^{-n/2},
\end{split}
\end{equation*}
which proves \eqref{3.28} as desired.\vspace{0.2cm}

\end{proof}

\begin{proof}[Proof of Theorem~\ref{thm:kernelbounds1} for high energies, $\blambda \geq 1$]

We express the spectral measure as a sum of 4 types of terms, (i)
--- (iv), as in Theorem~\ref{thm:specmeas}. Then, using
Proposition~\ref{prop:Qi}, we see that the microlocalized spectral
measure $Q_i(\blambda) \spmeas Q_i(\blambda)^*$ has microsupport
contained in $\Lambda^{nd}$, so the terms of type (iii) can be
disregarded. Clearly, terms of type (ii) and (iv) satisfy the
conclusion of Theorem~\ref{thm:kernelbounds1}, so it is only
necessary to consider the terms of type (i).

For terms of type (i), Proposition~\ref{prop:d<1} shows that
Theorem~\ref{thm:kernelbounds1} is satisfied. On the other hand, for large distance, we know from Proposition~\ref{prop:flowout-structure} that $\Lambda^{nd}_+$ and $\Lambda^{nd}_-$ both project diffeomorphically to a open set $V \subset X^2_0 \setminus \diagz$ under ${}^\Phi \pi$, and therefore, according to \cite[Proposition 20]{Chen-Hassell1}, the Lagrangian submanifolds $\Lambda^{nd}_\pm$ are parametrized by the distance function, $\pm d(z,z')$.
The amplitude has a classical expansion in powers of $h = \blambda^{-1}$, as shown by the construction of \cite[Section 4]{Chen-Hassell1}, and the leading power of $h$ is $h^{1/2 - (n+1)/2} = h^{-n/2}$, as claimed.

\end{proof}

\begin{corollary}\label{changed spectral measure estimates}Let $\delta$ be a small positive number. For high energies, $\blambda \geq 1$, one has
\begin{equation*}\begin{gathered}
\Big|  Q_k(\blambda)
\bigg( \big( \frac{d}{d\bsigma} \big)^j dE_{P}(\bsigma) \bigg) Q_k^\ast (\blambda)(z,z')  \Big| \leq
\begin{cases}
C \blambda^{n-j} (1 + d(z,z') \blambda)^{-n/2+j}, \text{ for } d(z,z') \leq 1 \\
C \blambda^{n/2} d(z,z')^{j} e^{-nd(z,z')/2}, \text{ for } d(z,z')
\geq 1
\end{cases}
\end{gathered},\end{equation*} provided $\bsigma$ is in $[(1 - \delta) \blambda, (1 + \delta) \blambda]$.\end{corollary}


\section{Restriction theorem at high energy}\label{sec:restriction}

We now prove the restriction theorem, Theorem~\ref{thm:restriction},
at high energies, $\blambda \geq 1$.

We apply complex interpolation to the analytic (in the parameter $a \in \CC$) family of operators
$$
\phi(P/\blambda)\chi_+^a(\blambda - P).
$$
Here $\phi$ is a smooth function supported on $[1 - \delta, 1 + \delta]$ and equal to $1$ on $[1 - \delta/2, 1 + \delta/2]$, whilst $\chi_+^a$ is an entire family of distributions, defined for $\Re a > -1$ by
$$
\chi_+^a = \frac{x_+^a}{\Gamma(a + 1)} \quad \mbox{with} \quad x_+^a = \bigg\{ \begin{array}{r@{\quad \mbox{if} \quad}l}x^a & x \geq 0\\ 0 & x < 0\end{array} ,  \quad \Re a > -1.
$$
When $\Re a > 0$, we have
$$\frac{d}{dx}\chi_+^a = \chi_+^{a - 1},$$
and using this identity, we extend $\chi_+^a$ to the entire complex $a$-plane. Since $\chi_+^0 = H(x)$, we have $\chi_+^{-1} = \delta_0$, and more generally $\chi_+^{-k} = \delta_0^{(k - 1)}$. Therefore,
\begin{equation}
\chi_+^0(\blambda - P) = E_{P}\big((0, \blambda]\big) \quad
\mbox{and} \quad \chi_+^{-k}(\blambda - P) =
\Big(\frac{d}{d\blambda}\Big)^{k - 1} dE_{P}(\blambda).
\label{chi-k}\end{equation}

 Moreover, for any $\mu,\nu \in \mathbb{C},$ it is shown in  \cite[p.86]{Hormander1} that  $$\chi_+^\mu \ast \chi_+^\nu = \chi_+^{\mu + \nu + 1}.$$
 Using this identity, and the fact that  the $\blambda$-derivatives of the spectral measure are well-defined and obey kernel estimates as in Theorem~\ref{thm:kernelbounds}, we define, following \cite{Guillarmou-Hassell-Sikora}, operators $\chi_+^a(\blambda - P)$. For $k  \in \NN$ and $-(k + 1) < \Re a < 0$, we define
 $$
 \chi_+^a(\blambda - P) = \chi_+^{k + a} \ast \chi_+^{-(k + 1)}(\blambda - P) = \int_0^\blambda \frac{\bsigma^{k + a}}{k + a + 1}\Big(\frac{d}{d\blambda}\Big)^kdE_{P}(\blambda - \bsigma)\,d\bsigma.
 $$

 A standard application of Stein's complex interpolation theorem \cite{interpolation} yields
 \begin{proposition}\label{prop:ci}
 Suppose that, for $s \in \RR$, we have
 $$
\big\| Q_i(\blambda) \phi(P/\blambda) \chi_+^{is}(\blambda - P) Q_i^*(\blambda) \big\|_{L^2(X) \to L^2(X)} \leq
C_1 e^{C(1 + |s|)},
$$
and for some $\beta > 0$,
 $$
\big\| Q_i(\blambda) \phi(P/\blambda) \chi_+^{-\beta + is}(\blambda - P) Q_i^*(\blambda)\big\|_{L^1(X) \to
L^\infty(X)} \leq C_2 e^{C(1 + |s|)}.
$$
Then, the spectral measure $\spmeas = \chi_+^{-1}(\blambda - P)$ is
bounded from $L^{p}(X) \to L^{p'}(X)$, $p = 2\beta/(\beta+1)$, with
an operator norm bound
\begin{equation}
\big\|  \spmeas \big\|_{L^p(X) \to L^{p'}(X)} \leq C'(C) \, C_1^{(\beta-1)/\beta} \, C_2^{1/\beta}  .
\end{equation}
\end{proposition}

Therefore, to prove Theorem~\ref{thm:restriction}, for $\blambda
\geq 1$,  we need to establish the estimates
\begin{equation}
\big\|Q_i(\blambda) \phi(P/\blambda) \chi^{is}(\blambda - P) Q_i^*(\blambda)\big\|_{L^2 \rightarrow L^2} \leq C_1
e^{C (1 + |s|)}, \label{22}\end{equation} and for $p \in [1,
2(n+2)/(n+4)]$, we require
\begin{equation}
\big\|Q_i(\blambda) \phi(P/\blambda) \chi^{-n/2 -1 + is}(\blambda - P) Q_i^*(\blambda)\big\|_{L^1 \rightarrow
L^\infty} \leq C_2 \blambda^{n/2} e^{C (1 + |s|)},
\label{psmall}\end{equation} while for $p \in [2(n+2)/(n+4), 2)$, we
require
\begin{equation}
\big\|Q_i(\blambda) \phi(P/\blambda) \chi^{-j -1+ is}(\blambda - P) Q_i^*(\blambda)\big\|_{L^1 \rightarrow
L^\infty} \leq C_2 \blambda^{n/2} e^{C (1 + |s|)}, \quad \text{for
all } j \in \ZZ, \ j \geq \frac{n}{2}. \label{plarge}\end{equation}

Estimate \eqref{22} follows immediately from the sup bound on the multiplier $\chi_+^{is}$:
$$
| \chi_+^{is}(t)| \leq \big| \frac1{\Gamma(is)} \big| \leq e^{\pi |s|/2}.
$$

For the remaining two estimates, we invoke  \cite[Lemma 3.3]{Guillarmou-Hassell-Sikora}, which we repeat here:

\begin{lemma}\label{infty}
Suppose that $k \in \NN$, that $-k < a<b<c$ and that $b=\theta a +(1-\theta)c$. Then there exists a constant  $C$ such that  for any $C^{k-1}$ function $f \colon \RR \to \CC$ with compact support,  one has
$$
\| \chi_+^{b+is}*f\|_\infty \le C(1+|s|)e^{\pi |s|/2} \| \chi_+^{a}*f\|^\theta_\infty  \| \chi_+^{c}*f\|^{1-\theta}_\infty
$$
for all $s\in \RR$.
\end{lemma}

Before proving \eqref{psmall} and \eqref{plarge}, we first rewrite $\phi(P/\blambda) \chi_+^{b + is}(\blambda - P)$ as a convolution. \begin{eqnarray*}\lefteqn{\phi(P/\blambda) \chi_+^{b + is}(\blambda - P)} \\ &=&\int \phi(\bsigma/\blambda)\chi_+^{b + is + k - 1}\ast \chi_+^{-k}(\blambda -\bsigma) dE_P(\bsigma) d\bsigma \\ &=& \iint \phi(\bsigma/\blambda) \chi_+^{b + is + k - 1}(\alpha) \chi_+^{-k} (\blambda - \bsigma - \alpha) dE_P(\bsigma) \,d\bsigma d\alpha \\ &=& \blambda^{b + is + 1} \iint \phi(\bsigma) \chi_+^{b + is + k - 1}(\alpha) \chi_+^{-k} (1 - \bsigma - \alpha) dE_P(\blambda\bsigma) \,d\bsigma d\alpha \\ &=& \blambda^{b + is + 1} \int \chi_+^{b + is + k - 1}(\alpha)    \frac{d^{k - 1}}{d \bsigma^{k - 1}}\big(\phi(\bsigma)  dE_P(\blambda\bsigma)\big)\bigg|_{\bsigma = 1 - \alpha} \,d\bsigma d\alpha  \\ 
&=& \blambda^{b + is + 1} \Big(  \chi_+^{b + is + k - 1}\ast   \big(\phi(\cdot)  dE_P(\blambda\cdot)\big)^{(k - 1)} \Big) (1).  \end{eqnarray*}

To prove \eqref{psmall}, we rewrite the function $\chi_+^{- n/2 - 1 + is}$ as 
$$
\chi_+^{-n/2 -1 + is} = \left\{\begin{array}{ll}\chi_+^{-3/2 + is} \ast \chi_+^{-k} & n + 1 = 2k\\ \chi_+^{-2 + is} \ast \chi_+^{-k} & n + 1 = 2k + 1\end{array}\right.
$$

Then we microlocalize the operator $\phi(\sqrt{L}/\blambda)\chi_+^{-n/2 - 1 + is}(\blambda - \sqrt{L})$ and apply Lemma \ref{infty}. In the following calculations, the operators (that is, their Schwartz kernels) are evaluated at the point $(z, z')$, which we do not always indicate in notation. 

\begin{eqnarray*}
\lefteqn{\bigg|Q(\blambda)\phi(\sqrt{L}/\blambda)\chi_+^{-n/2 - 1 + is}(\blambda - \sqrt{L})Q^*(\blambda)\bigg|}\\ 
&=& \bigg|\blambda^{-k + is}\int \chi_+^{- 2 + is}(\alpha) Q(\blambda) \frac{d^{k - 1}}{d \bsigma^{k - 1}}\Big(\phi(\bsigma) dE_{\sqrt{L}}(\blambda\bsigma)\Big)\bigg|_{\bsigma = 1 - \alpha} Q^*(\blambda)\,  d\alpha \bigg|\\ 
&\leq& C_s\blambda^{-k} \sup_{\Lambda}\bigg|\int \chi_+^{- 1}(\alpha) Q(\blambda) \frac{d^{k - 1}}{d \bsigma^{k - 1}}\Big(\phi(\bsigma) dE_{\sqrt{L}}(\blambda\bsigma)\Big)\bigg|_{\bsigma = \Lambda - \alpha} Q^*(\blambda)\,  d\alpha\bigg|^{1/2}\\ 
& & \quad\quad \times \sup_{\Lambda}\bigg|\int \chi_+^{- 3}(\alpha) Q(\blambda) \frac{d^{k - 1}}{d \bsigma^{k - 1}}\Big(\phi(\bsigma) dE_{\sqrt{L}}(\blambda\bsigma)\Big)\bigg|_{\bsigma = \Lambda - \alpha} Q^*(\blambda)\,  d\alpha\bigg|^{1/2} 
\end{eqnarray*}
\begin{eqnarray*}
&\leq& C_s\blambda^{-n/2} \sup_{\Lambda}\bigg| Q(\blambda) \frac{d^{k - 1}}{d \bsigma^{k - 1}}\Big(\phi(\bsigma) dE_{\sqrt{L}}(\blambda\bsigma)\Big)\bigg|_{\bsigma = \Lambda} Q^*(\blambda)\bigg|^{1/2}\\ 
& & \quad\quad \quad\quad \times \sup_{\Lambda}\bigg| Q(\blambda) \frac{d^{k + 1}}{d \bsigma^{k + 1}}\Big(\phi(\bsigma) dE_{\sqrt{L}}(\blambda\bsigma)\Big)\bigg|_{\bsigma = \Lambda } Q^*(\blambda)\bigg|^{1/2}. 
\end{eqnarray*}

We now plug in Corollary \ref{changed spectral measure estimates} and get
\begin{eqnarray*}\lefteqn{\bigg|Q(\blambda)\phi(\sqrt{L}/\blambda)\chi_+^{-n/2 - 1 + is}(\blambda - \sqrt{L})Q^*(\blambda)\bigg|} \\&\leq&C_s \blambda^{n/2} \sup_{\Lambda\in \text{supp}\phi}\bigg|\sum_{l = 0 \dots k - 1} \Lambda^{n - l} (1 + \blambda \Lambda d(z, z'))^{- n/2 + l}\bigg|^{1/2}\\ & & \quad\quad \quad\quad\,\sup_{\Lambda \in \text{supp}\phi}\bigg|\sum_{l = 0 \dots k + 1}\Lambda^{n - l}(1 + \blambda\Lambda d(z, z'))^{- n/2 + l}\bigg|^{1/2}\\ &\leq&C_s \blambda^{n/2},\end{eqnarray*} provided $d(z, z')$ is small. On the other hand, if $d(z, z')$ is large, \begin{eqnarray*}
\lefteqn{\bigg|Q(\blambda)\phi(\sqrt{L}/\blambda)\chi_+^{-n/2 - 1 + is}(\blambda - \sqrt{L})Q^*(\blambda)\bigg|}\\ &\leq& C_s  \sup_{\Lambda\in \text{supp}\phi}\bigg|\sum_{l = 0 \dots k - 1} \Lambda^{n/2 - l} d(z, z')^le^{-nd(z, z')}\bigg|^{1/2}\bigg| \sum_{l = 0 \dots k + 1} \Lambda^{n/2 - l} d(z, z')^le^{-nd(z, z')}\bigg|^{1/2}\\ &\leq&C_s \blambda^{n/2}. \end{eqnarray*}

When $n = 2k$, the proof is identical, apart from using the other expression of the operator.\begin{eqnarray*}
\phi(\sqrt{L}/\blambda)\chi_+^{-n/2 -1 + is}(\blambda - \sqrt{L}) = \blambda^{- n/2 + is}\int \chi_+^{- 3/2 + is}(\alpha)  \frac{d^{k - 1}}{d \bsigma^{k - 1}}\Big(\phi(\bsigma) dE_{\sqrt{L}}(\blambda\bsigma)\Big)\bigg|_{\bsigma = 1 - \alpha}\, d\alpha .\end{eqnarray*}
We use the same argument and get 																																																			\begin{eqnarray*}
\lefteqn{\bigg|Q(\blambda)\phi(\sqrt{L}/\blambda)\chi_+^{- n/2 - 1 + is}(\blambda - \sqrt{L})Q^*(\blambda)\bigg|}\\ &=& \bigg|\blambda^{- n/2 + is}\int \chi_+^{- 3/2 + is}(\alpha) Q(\blambda) \frac{d^{k - 1}}{d \bsigma^{k - 1}}\Big(\phi(\bsigma) dE_{\sqrt{L}}(\blambda\bsigma)\Big)\bigg|_{\bsigma = 1 - \alpha} Q^*(\blambda)\,  d\alpha \bigg|\\  &\leq& C_s\blambda^{-n/2} \sup_{\Lambda}\bigg| Q(\blambda) \frac{d^{k - 1}}{d \bsigma^{k - 1}}\Big(\phi(\bsigma) dE_{\sqrt{L}}(\blambda\bsigma)\Big)\bigg|_{\bsigma = \Lambda} Q^*(\blambda)\bigg|^{1/2}\\ & & \quad\quad \quad\quad\times\sup_{\Lambda}\bigg| Q(\blambda) \frac{d^{k }}{d \bsigma^{k}}\Big(\phi(\bsigma) dE_{\sqrt{L}}(\blambda\bsigma)\Big)\bigg|_{\bsigma = \Lambda } Q^*(\blambda)\bigg|^{1/2} \\ &\leq&C_s \blambda^{n/2}.\end{eqnarray*}

To prove \eqref{plarge},  we rewrite the function $\chi_+^{- j - 1 + is}$ as 
$$
\chi_+^{- j -1 + is} = \chi_+^{-2 + is} \ast \chi_+^{- j} .
$$
Then we again apply Lemma \ref{infty} to the microlocalized operator $Q(\blambda)\phi(\sqrt{L}/\blambda)\chi_+^{- j - 1 + is}(\blambda - \sqrt{L})Q^*(\blambda)$ and get\begin{eqnarray*}
\lefteqn{\bigg|Q(\blambda)\phi(\sqrt{L}/\blambda)\chi_+^{- j - 1 + is}(\blambda - \sqrt{L})Q^*(\blambda)\bigg|}\\ &=& \bigg|\blambda^{- j + is}\int \chi_+^{- 2 + is}(\alpha) Q(\blambda) \frac{d^{j - 1}}{d \bsigma^{j - 1}}\Big(\phi(\bsigma) dE_{\sqrt{L}}(\blambda\bsigma)\Big)\bigg|_{\bsigma = 1 - \alpha} Q^*(\blambda)\,  d\alpha \bigg|\\ &\leq& C_s\blambda^{-j} \sup_{\Lambda}\bigg|\int \chi_+^{- 1}(\alpha) Q(\blambda) \frac{d^{j - 1}}{d \bsigma^{j - 1}}\Big(\phi(\bsigma) dE_{\sqrt{L}}(\blambda\bsigma)\Big)\bigg|_{\bsigma = \Lambda - \alpha} Q^*(\blambda)\,  d\alpha\bigg|^{1/2}\\ & & \quad\quad\times\sup_{\Lambda}\bigg|\int \chi_+^{- 3}(\alpha) Q(\blambda) \frac{d^{j - 1}}{d \bsigma^{j - 1}}\Big(\phi(\bsigma) dE_{\sqrt{L}}(\blambda\bsigma)\Big)\bigg|_{\bsigma = \Lambda - \alpha} Q^*(\blambda)\,  d\alpha\bigg|^{1/2} \\ &\leq& C_s\blambda^{- j} \sup_{\Lambda}\bigg| Q(\blambda) \frac{d^{j - 1}}{d \bsigma^{j - 1}}\Big(\phi(\bsigma) dE_{\sqrt{L}}(\blambda\bsigma)\Big)\bigg|_{\bsigma = \Lambda} Q^*(\blambda)\bigg|^{1/2}\\ & & \quad\quad \quad\quad\times\sup_{\Lambda}\bigg| Q(\blambda) \frac{d^{j + 1}}{d \bsigma^{j + 1}}\Big(\phi(\bsigma) dE_{\sqrt{L}}(\blambda\bsigma)\Big)\bigg|_{\bsigma = \Lambda } Q^*(\blambda)\bigg|^{1/2}. 
\end{eqnarray*}

As in the first case, the spectral measure estimates give
\begin{eqnarray*}
\lefteqn{\bigg|Q(\blambda)\phi(\sqrt{L}/\blambda)\chi_+^{- j - 1 + is}(\blambda - \sqrt{L})Q^*(\blambda)\bigg|}\\ &\leq&C_s \blambda^{n - j} \sup_{\Lambda\in \text{supp}\phi}\bigg|\sum_{l = 0 \dots j - 1} \Lambda^{n - l} (1 + \blambda \Lambda d(z, z'))^{- n/2 + l}\bigg|^{1/2}
\\ & & \quad\quad \quad\quad\times\sup_{\Lambda \in \text{supp}\phi}\bigg|\sum_{l = 0 \dots j + 1}\Lambda^{n - l}(1 + \blambda\Lambda d(z, z'))^{- n/2 + l}\bigg|^{1/2}\\ 
&\leq&C_s \blambda^{n/2},\end{eqnarray*} 
provided $d(z, z')$ is small. On the other hand, if $d(z, z')$ is large, \begin{eqnarray*}
\lefteqn{\bigg|Q(\blambda)\phi(\sqrt{L}/\blambda)\chi_+^{- j - 1 + is}(\blambda - \sqrt{L})Q^*(\blambda)\bigg|}\\ &\leq& C_s \blambda^{n/2 - j} \sup_{\Lambda\in \text{supp}\phi}\bigg|\sum_{l = 0 \dots j - 1} \Lambda^{n/2 - l} d(z, z')^le^{-nd(z, z')}\bigg|^{1/2}\bigg| \sum_{l = 0 \dots j + 1} \Lambda^{n/2 - l} d(z, z')^le^{-nd(z, z')}\bigg|^{1/2}\\ &\leq&C_s \blambda^{n/2} .\end{eqnarray*}

\begin{remark}
Notice that it is here that we gain an advantage by working on an
asymptotically hyperbolic rather than conic space: the exponential
decay $e^{-nd/2}$ in the large distance estimate kills the polynomial growth $d^{j}$ caused by $j$
differentiations of the phase function $e^{\pm i\blambda d}$, so
there is no limit to the number of differentiations that we can
consider. On an asymptotically conic manifold, however, if we
differentiate more than $(\dim X - 1)/2$ times, we get a growing
kernel as $d(z,z') \to \infty$, and no $L^1 \to L^\infty$ estimate
is possible.
\end{remark}


\section{Spectral multipliers}\label{sec:sm}

In this section we prove Theorem~\ref{thm:sm}, assuming that $(X^\circ, g)$ is a Cartan-Hadamard manifold, as well as being asymptotically hyperbolic and nontrapping, with no resonance at the bottom of the continuous spectrum.

\subsection{A geometric lemma}\label{subsec:meas}
In order to adapt the proof from Section~\ref{sec:model}, we need to establish comparability between the Riemannian measure on hyperbolic space, and the Riemannian measure on $(X^\circ, g)$, as expressed in polar coordinates. Recall that on a Cartan-Hadamard manifold, the exponential map from $T_p X$, $p \in X$ to $X$ is a diffeomorphism from $T_pX$ to $X$. Thus, the metric on $X$ can be expressed globally in polar normal  coordinates based at $p$. Let $r$ be the distance, and $\omega \in S^n$, be polar normal  coordinates based at $p$.

\begin{lemma}\label{lem:measure}
Let $X$ be an asymptotically hyperbolic Cartan-Hadamard manifold, and let $p \in X$ be any point.
The Riemannian measure on $X$ can be expressed in the form
$$
 m_p(r, \omega) (\sinh r)^n dr d\omega,
 $$
 where $m_p(r, \omega)$ is uniformly bounded on $X \times X$ (that is, uniform in $p$ as well as in $(r, \omega)$).
 \end{lemma}

 \begin{proof}
 This result can be extracted from the resolvent construction in \cite{Chen-Hassell1}.
 Recall that in that paper, the outgoing resolvent $(h^2 \Delta - h^2 n^2/4 - (1 - i0))^{-1}$ was shown to be a sum of terms, the principal one of which is a semiclassical intersecting Lagrangian distribution $I^{1/2}(X^2_0, (N^* \diagz, \Lambda_+); \Omegazh)$. Here $\Lambda_+ = \Lambda_+^{nd} \cup \Lambda_+^*$ is the closure of the forward bicharacteristic relation, in a certain sense. In the case of a Cartan-Hadamard manifold, the projection ${}^\Phi \pi : {}^\Phi T^* X^2_0 \to X^2_0$ restricts to a diffeomorphism from $\Lambda_+ \setminus N^* \diagz$ to $X^2_0 \setminus \diagz$; that is, except over the diagonal, $\Lambda_+$ projects diffeomorphically to the base $X^2_0$. We also point out that there is no need to decompose $\Lambda_+$ into pieces $\Lambda_+^{nd} \cup \Lambda_+^*$ as was done in \cite{Chen-Hassell1} to deal with geodesics that might `return' to the front face  $\FF$; this is not possible for Cartan-Hadamard manifolds.

 The Lagrangian $\Lambda_+$ can be given coordinates as follows: first, we use coordinates $(z', \omega)$ for $\Lambda_+ \cap N^* \diagz$, where $z'$ is a coordinate in $X^\circ$ (corresponding to the right variable in $X^2_0$) and  $\omega \in S^n$ is a coordinate on the unit tangent bundle in $T_{z'} X^\circ$, with respect to the metric $g$. Then by definition $\Lambda_+$ is the flowout from $\Lambda_+ \cap N^* \diagz$ by bicharacteristic flow, which coincides with geodesic flow in this case. Let $r$ denote the function on $\Lambda_+$ equal to the time taken to flow to that point from $\Lambda_+ \cap N^* \diagz$ by the left geodesic flow. This gives us $(z', r, \omega)$ as coordinates on $\Lambda_+$. Then, using the projection ${}^\Phi \pi$ to the base, $(r, \omega)$ may be identified with polar normal coordinates based at $z'$.

 Now consider the principal symbol at $\Lambda_+$. By \cite{Chen-Hassell1}, if we use coordinates $(z, z')$ arising from $X \times X$ on $\Lambda_+$ (away from $N^* \diagz$), then the principal symbol is  $\sim (\rho_L \rho_R)^{n/2}$ times  $|dg(z) dg'(z')|^{1/2}$, where $dg$ ($dg'$) indicate the Riemannian measure in the left (right) variables, and we use the notation $a \sim b$ to mean that $C^{-1} b \leq a \leq C b$ for some uniform $C$. Next we recall from Proposition~\ref{prop:dist} that the distance $r$ on $X^2_0$ is such that
$e^{-nr/2} \sim (\rho_L \rho_R)^{n/2}$ for $r \geq 1$.  It follows that the principal symbol is comparable to
\begin{equation}
e^{-nr/2} |dg dg'|^{1/2}
\label{prs1}\end{equation}
on $\Lambda_+$, for $r \geq 1$.

On the other hand, the principal symbol  $a$ satisfies the transport equation
$$
\mathcal{L}_{\partial_r} a = 0,
$$
in the coordinates $(z', r, \omega)$. Since $a$ is a half-density,  it must take the form
$$
\Big| b(z', \omega) dz' dr d\omega \Big|^{1/2}.
$$
We can compute $b(z', \omega)$  by comparing with the symbol of the resolvent at $N^* \diag$. Using coordinates $(z', \omega, \tau)$, where $\tau$ is the norm on $T^*_{z'} X$ with respect to the metric $g$ (that is, $(\omega, \tau)$ are polar coordinates in $T^*_{z'} X$), this symbol is $(\tau^2 - 1)^{-1}|dg' \tau^n d\tau d\omega|^{1/2}$. A simple calculation shows that ${\tau, r} = 1$.  Using \cite[Equation (B.2)]{Chen-Hassell1}, we find that
\begin{equation}
a = c \,  \Big| dg' dr d\omega \Big|^{1/2}, \quad c > 0 \text{ constant.}
\label{prs2}\end{equation}
Comparing \eqref{prs1} and \eqref{prs2}, we find that
$$
dg \sim e^{nr} dr d\omega \text{ for } r \geq 1,
$$
which completes the proof.
\end{proof}

So, let $F \in H^s([-1, 1])$, $s > (n+1)/2$, be an even function. We consider the operator $F(\alpha P)$, where $\alpha \in (0, 1]$. To analyze this operator, we break the Schwartz kernel into two pieces using the characteristic function $\chi_{d(z,z') \leq 1}$. The near-diagonal piece $F(\alpha P)\chi_{d(z,z') \leq 1}$ can be treated using the methods from \cite{Guillarmou-Hassell-Sikora}; this operator essentially satisfies Theorem \ref{restriction2multiplier}. The far-from-diagonal piece, $F(\alpha P)(1-\chi_{d(z,z') \leq 1})$, can be treated rather like the case of hyperbolic space studied in Section~\ref{sec:model}.

\subsection{Near diagonal part of $F(\alpha P)$}\label{subsec:dist<1}
Theorem~\ref{restriction2multiplier} does not apply directly in the current setting, since the volume of balls of radius $\rho$ are not comparable to $\rho^{n+1}$ for large $\rho$ on asymptotically hyperbolic manifolds; instead,  the volume grows as $e^{n\rho}$ as $\rho \to \infty$. However, it is certainly the case that the volume of balls of radius $\rho \leq 1$ is comparable to $\rho^{n+1}$. This follows from the Bishop-Gromov inequality: if the sectional curvatures are between $0$ and $-\kappa$, say, then the volume of any ball of radius $\rho$ is bounded by the volume in Euclidean space, and the volume on a simply connected space of constant curvature $-\kappa$.

The place where this volume comparability was used in \cite{Guillarmou-Hassell-Sikora} was in the proof of the following Lemma, which we modify so as to apply to our near-diagonal operator.

\begin{lemma}[{\cite[Lemma 2.7]{Guillarmou-Hassell-Sikora}}]
\label{lem:doubling}
Suppose that $(X, d, \mu)$ is a metric measure space, with metric $d$ and doubling measure $\mu$, such that the balls of radius $\rho \leq 1$ have measure comparable to $\rho^{n+1}$. Assume that $S$ is an integral operator, bounded from $L^p(X)$ to $L^q(X)$ for some $1 \leq p < q \leq \infty$. Let $S \chi_{d(z,z') \leq s}$, be the integral operator given by the integral kernel of $S$ times the characteristic function of $\{ (z,z') \mid  d(z,z') \leq s \}$, for some $s \leq 1$. Then $$\|S \chi_{d(z,z') \leq s}\|_{L^p \rightarrow L^p} \leq C s^{(n + 1)(1/p - 1/q)} \|S\|_{L^p \rightarrow L^q}.$$
\end{lemma}

\begin{proof} We omit the proof, which is a trivial modification of the proof of \cite[Lemma 2.7]{Guillarmou-Hassell-Sikora}.
\end{proof}

Using this lemma we prove a modified version of Theorem~\ref{restriction2multiplier} in an abstract setting.

\begin{proposition}
Let $(X, d, \mu)$ be as in Lemma~\ref{lem:doubling}.
Suppose $\Delta$ is a positive self-adjoint operator  with finite propagation speed on $L^2(X)$. If the restriction estimate
\begin{equation}
\|dE_{\sqrt{\Delta}}(\blambda)\|_{L^p \rightarrow L^{p^\prime}} \leq \left\{ \begin{array}{ll}  C  &  \mbox{when $\blambda$ is small,} \\  C \blambda^{(n + 1)(1/p - 1/p') - 1}  &  \mbox{when $\blambda$ is large} \end{array} \right.
\label{Cre}\end{equation}
 holds for $1 \leq p \leq 2(n + 2)/(n + 4)$, then spectral multipliers {localized near the diagonal} are uniformly bounded in $0 < \alpha < 1$, in the sense $$\sup_{0 < \alpha < 1} \|F(\alpha \sqrt{\Delta}){\chi_{d(z,z') \leq 1}}\|_{L^p \rightarrow L^p} \leq C \|F\|_{H^s},$$ where $F \in H^s(\mathbb{R})$ is an even function  with $s > (n + 1) (1/p - 1/2)$ supported in $[-1, 1]$.\end{proposition}

\begin{proof}
We follow the proof of \cite[Section 2]{Guillarmou-Hassell-Sikora}.
Suppose $\eta$ is an even smooth function compactly supported on
$(-4, 4)$, satisfying $$\sum_{l \in \mathbb{Z}} \eta(2^{- l} t) = 1
\quad \mbox{for all $t \neq 0$.}$$ Thus we take a partition of unity
for $F(\blambda)$, say $F(\blambda) = F_0 + \sum_{l > 0}
F_l(\blambda)$, where \begin{eqnarray*}F_0(\blambda) & = &
\frac{1}{2\pi} \int_{- \infty}^\infty \sum_{l \leq 0} \eta(2^{- l}
t) \hat{F}(t) \cos(t \blambda) \, dt\\ F_l(\blambda) & = &
\frac{1}{2\pi} \int_{- \infty}^\infty \eta(2^{- l} t) \hat{F}(t)
\cos(t \blambda) \, dt \quad \quad \mbox{for $l > 0$}\end{eqnarray*}
By virtue of finite speed of propagation of $\cos(t P)$
\cite{C-G-T}, i.e. $$\text{supp}\,\cos(tP) \subset \{d(z, z^\prime)
\leq |t|\},$$ the kernel of $F_l(\alpha P) \chi_{d(z,z') \leq 1}$ is
supported on $$\{d(z, z^\prime) \leq 2^{l + 2} \alpha\}$$ as
$\eta(2^{-l} t)$ is supported on $(- 2^{l + 2}, 2^{l + 2})$.

 By Lemma~\ref{lem:doubling},
 \begin{multline}
 \|F(\alpha P)\chi_{d(z,z') \leq 1}\|_{L^p \rightarrow L^p} \leq \sum_{l \geq 0} \|F_l(\alpha P)\chi_{d(z,z') \leq 1}\|_{L^p \rightarrow L^p} \\ \leq C \sum_{l \geq 0} (2^l \alpha)^{(n + 1)(1/p - 1/2)} \|F_l(\alpha P)\|_{L^p \rightarrow L^2}.
\label{ldecomp} \end{multline}
  We take a further decomposition $$F_l(\alpha P) = \psi F_l(\alpha P) + (1 - \psi) F_l(\alpha P)$$ by a cutoff function $\psi$ supported on $(-4, 4)$ such that $\psi(\blambda) = 1$ for $\blambda \in (-2, 2)$.

Then a  $T^\ast T$ argument reduces $\|\psi F_l(\alpha P)\|_{L^p \rightarrow L^2}$ to the  restriction estimates. \begin{eqnarray*}\|\psi F_l(\alpha P)\|^2_{L^p \rightarrow L^2} &=& \||\psi F_l|^2 (\alpha P)\|_{L^p \rightarrow L^{p^\prime}} \\&\leq& \int_0^{4/\alpha} |\psi F_l(\alpha \blambda)|^2 \|dE_{P}(\blambda)\|_{L^p \rightarrow L^{p^\prime}} \, d\blambda \\&\leq& \frac{C}{\alpha} \int_0^4 |\psi F_l(\blambda)|^2 \|dE_{P} (\blambda/\alpha)\|_{L^p \rightarrow L^{p^\prime}} \, d\blambda\\
\leq \ \frac{C}{\alpha} \int_0^{\alpha}   |\psi F_l(\blambda)|^2  \,
d\blambda &+& \frac{C}{\alpha} \int_{\alpha}^4 |\psi
F_l(\blambda)|^2 \bigg( \frac{\blambda}{\alpha} \bigg)^{(n + 1)(1/p
- 1/p^\prime) - 1} \, d\blambda    ,
\end{eqnarray*}
where we used \eqref{Cre} in the last line. So
$$
\|\psi F_l(\alpha P)\|_{L^p \rightarrow L^2} \leq C \alpha^{-((n+1)/2)(1/p - 1/p')} \| \psi F_l \|_2 = C \alpha^{-(n+1)(1/p - 1/2)} \| \psi F_l \|_2.
$$
We obtain \begin{eqnarray} \sum_{l \geq 0} (2^l \alpha)^{(n + 1) (1/p -1/2)} \|\psi F_l(\alpha P)\|_{L^p \rightarrow L^2} & \leq& \sum_{l \geq 0} 2^{l(n + 1)(1/p - 1/2)} \|\psi F_l\|_2\\ &\leq& C \|F\|_{B_{1, 2}^{(n + 1)(1/p - 1/2)}}\\&\leq& C \|F\|_{H^s} \quad \quad \mbox{for $s > (n + 1)(1/p - 1/2)$}.
\label{psiFl}\end{eqnarray}

We next treat the terms involving $(1 - \psi) F_l$. This works exactly as in \cite[Section 2]{Guillarmou-Hassell-Sikora}.

Using restriction estimates as above, we have
\begin{align*}
\big\|(1 - \psi) F_l (\alpha P)\big\|^2_{L^p \rightarrow L^2} &=  \big\| \, |(1 - \psi) F_l|^2 (\alpha P)\big\|_{L^p \rightarrow L^{p^\prime}} \\
&\leq    \frac{C}{\alpha} \int_{2}^\infty \Big|(1 - \psi)(\blambda)
F_l(\blambda)\Big|^2 \bigg(\frac{\blambda}{\alpha}\bigg)^{(n +
1)(1/p - 1/p^\prime)-1} \, d\blambda
\end{align*}
where we used the fact that $\blambda \geq 2$ on the support of $1 -
\psi$. Note that
\begin{equation}
(1-\psi(\blambda)) F_l(\blambda) = \frac{1-\psi(\blambda)}{2\pi}
\int_\mathbb{R} \int_0^1 e^{it(\blambda - \blambda^\prime)}
\eta(2^{-l} t) F(\blambda^\prime) \, d\blambda^\prime dt
\label{ibp}\end{equation} and
$$
e^{it(\blambda - \blambda^\prime)} = \frac1{i^N(\blambda -
\blambda^\prime)^{N}} \frac{d^N}{dt^N} e^{it(\blambda -
\blambda^\prime)},
$$
where $\blambda - \blambda^\prime \geq \blambda/2$ for $\blambda \in
\supp 1 - \psi$ and $\blambda' \in \supp F$. Using this identity in
\eqref{ibp} and integrating by parts in $t$ yields
$$
\Big|\big((1 - \psi) F_l\big)(\blambda)\Big| \leq C \blambda^{- N}
2^{- N(l - 1)} \|F\|_2
$$
for any $N \in \mathbb{Z}_+$. Taking $N$ sufficiently large, we obtain
\begin{equation}
\sum_l (2^l \alpha)^{(n + 1)(1/p - 1/2)} \Big\|\big((1 - \psi) F_l\big)(\alpha P)\Big\|_{L^p \rightarrow L^2} \leq C \|F\|_{L^2} \leq C \|F\|_{H^s}.
\label{1-psiFl}\end{equation}
Combining \eqref{psiFl} and \eqref{1-psiFl} yields
$$
\sum_l (2^l \alpha)^{(n + 1)(1/p - 1/2)} \big\| F_l (\alpha P)\big\|_{L^p \rightarrow L^2} \leq  C \|F\|_{H^s},
$$
and together with \eqref{ldecomp} this proves the Proposition.

\end{proof}

\subsection{Away from the diagonal on asymptotically hyperbolic manifolds}
It remains to treat the kernel $F(\alpha P) \chi_{d(z,z') \geq 1}$. We will show that $F(\alpha P) \chi_{d(z,z') \geq 1}$ maps $L^p + L^2$ to $L^2$, with an operator norm uniform in $\alpha \in (0, 1]$. It suffices to show that $F(\alpha P) \chi_{d(z,z') \geq 1}$ maps $L^2 \to L^2$ , and 
$L^1 \to L^2$, with operator norms uniform in $\alpha$. The first statement follows from the fact that $F \in H^{(n+1)/2} \implies F \in L^\infty$, together with the result of Section~\ref{subsec:dist<1}. {In fact, we have proved $$\sup_{0 < \alpha < 1} \|F(\alpha \sqrt{\Delta}) \chi_{d(z, z') \leq 1}\|_{L^p \rightarrow L^p} \leq C \|F\|_{H^{(n+1)/2}}$$ provided $1 \leq p \leq 2(n + 2)/(n + 4)$. Noting the spectral multiplier is symmetric, we conclude that this operator is $L^p$ bounded for $1 < p < \infty$. Consequently, $F(\alpha \sqrt{\Delta}) \chi_{d(z, z') \geq 1}$ is $L^2$ bounded.} So in the remainder of this subsection, we show boundedness from $L^1$ to  $L^2$, with an operator norm uniform in $\alpha$.

Let $K_\alpha(z, z')$ denote the Schwartz kernel of $F(\alpha P)\chi_{d(z,z') \geq 1} $.
By Minkowski's inequality, the $L^1 \to L^2$ operator norm is bounded by
 $$\sup_{z^\prime} \bigg(\int \big|K_\alpha(z, z^\prime)\big|^2 \,d\mu_z\bigg)^{1/2}.$$
 Using the spectral theorem we have $$K(z, z^\prime) = \int_0^\infty F(\blambda) dE_{P}(\blambda)(z,z') \  \cdot \chi_{\{d(z, z^\prime) > 1\}}(z, z^\prime).$$

We use coordinates $(z', r, \omega)$ as in Section~\ref{subsec:meas}. Using Lemma~\ref{lem:measure}, we may estimate the Riemannian measure by $C e^{nr} dr d\omega$. Therefore, it suffices to bound
$$
\int_{\{r > 1\}} \bigg|\int_{0}^\infty F(\alpha \blambda)
dE_{P}(\blambda)(r, \omega, z^\prime) \, d\blambda \bigg|^2 e^{nr}
\, drd\omega.
$$

Using \eqref{bpm}, we expand the kernel of the spectral measure as follows:
\begin{equation}\begin{gathered}
dE_P(\blambda)(z,z') = \sum_{\pm}  e^{\pm i\blambda r} \Big( \sum_{j=0}^{[n/2]} \blambda^{n/2-j} b_{\pm,j}(z', r, \omega) e^{-nr/2} + c(\blambda, z', r, \omega) e^{-nr/2} \Big) \\
+    (\rho_L \rho_R)^{n/2+i\blambda} \, a_+ + (\rho_L
\rho_R)^{n/2-i\blambda} \, a_- + (xx')^{n/2+i\blambda} \, \tilde a_+
+ (xx')^{n/2-i\blambda} \, \tilde  a_-  
\end{gathered} \label{sm-exp}\end{equation}
where $b_{\pm,j}$ and $c$ are bounded, and where $a_\pm, \tilde
a_\pm$ are as in Theorem~\ref{thm:kernelbounds1}. Here, $c$ is
smooth in $\blambda$ at $\blambda = 0$ (due to our assumption that
the resolvent kernel is holomorphic at the bottom of the spectrum),
and decays as $O(\blambda^{-1/2})$ as $\blambda \to \infty$ for $n$
odd, or $O(\blambda^{-1})$ as $\blambda \to \infty$ for $n$ even.
Moreover, $c$ obeys symbolic estimates as $\blambda \to \infty$, so
$|d_\blambda c | = O(\blambda^{-3/2})$ as $\blambda \to \infty$ when
$n$ is odd, or $O(\blambda^{-2})$ when $n$ is even.

We now consider a single term $b_{\pm, j}$  in \eqref{sm-exp}. Thus, we need to estimate
\begin{equation}
\int_{\{r > 1\}} \bigg|\int_{0}^\infty F(\alpha \blambda) e^{\pm i
\blambda r} \blambda^{n/2-j} b_{\pm, j}(r, \omega, z^\prime)
e^{-nr/2} \, d\blambda \bigg|^2 e^{nr} \, drd\omega
\label{bj}\end{equation} uniformly in $\alpha$ and $z'$. Arbitrarily
choosing the sign $+$, using the uniform boundedness of $b_{j,
\pm}$, and simplifying, it is enough to uniformly bound
\begin{equation}
\int_{\{r > 1\}} \bigg|\int_{0}^\infty F(\alpha \blambda) e^{ i
\blambda r} \blambda^{n/2-j}   \, d\blambda \bigg|^2  \, dr.
\label{bj2}\end{equation}

To estimate this, we prove the following lemma.
\begin{lemma}\label{lem:FG}
Suppose that $F \in H^{(n+1)/2}(\RR)$ and $G(\blambda) =
\theta(\blambda) \blambda^m \phi(\blambda)$, where $\phi \in
C_c^\infty(\RR)$, $\theta$ is the Heaviside function, and where $0 <
m \leq n/2$. Then $\hat F * \hat G$ satisfies
\begin{equation}
\int_{r \geq R} \big| (\hat F * \hat G)(r) \big|^2 \, dr = O(R^{-(2m+1)}).
\end{equation}
\end{lemma}

\begin{proof}
We first observe that $|\hat G(r)| \leq \ang{r}^{-m-1}$. Indeed,
since the function $\theta(\blambda) \blambda^{m}$ is homogeneous of
degree $m$, the Fourier transform is homogeneous of degree $-1-m$,
and hence is $O(\ang{r}^{-1-m})$ as $r \to \infty$. The Fourier
transform $\hat G$ is therefore this homogeneous function convolved
with $\hat \phi$. As $\hat \phi \in \mathcal{S}(\RR)$,  $\hat G$ is
$L^\infty$, and still decays as $O(\ang{r}^{-1-m})$ as $r \to
\infty$.

It therefore suffices to show that
$$
\int_{r \geq R} \Big| \int |\hat F(r-s)| \ang{s}^{-m-1} \, ds \Big|^2 \, dr = O(R^{-(2m+1)}).
$$
We break the RHS into
\begin{equation}
\int_{r \geq R} \bigg| \int_{|s| \leq r/2} |\hat F(r-s)| \ang{s}^{-m-1} \, ds + \int_{|s| \geq r/2} |\hat F(r-s)| \ang{s}^{-m-1} \, ds \ \bigg|^2 \, dr.
\end{equation}
Using the inequality $(a+b)^2 \leq 2(a^2 + b^2)$, we estimate this by
\begin{equation}
2 \int_{r \geq R} \bigg| \int_{|s| \leq r/2} |\hat F(r-s)| \ang{s}^{-m-1} \, ds  \ \bigg|^2 \, dr +
2 \int_{r \geq R} \bigg|  \int_{|s| \geq r/2} |\hat F(r-s)| \ang{s}^{-m-1} \, ds \ \bigg|^2 \, dr.
\label{twoterms}\end{equation}
The first of these terms we treat as follows. We apply Cauchy-Schwarz to the inner integral, obtaining
\begin{equation}
2 \int_{r \geq R} \bigg(  \int_{|s| \leq r/2} \ang{s}^{-m-1} \bigg) \bigg( \int_{|s'| \leq r/2} |\hat F(r-s')|^2 \ang{s'}^{-m-1} \, ds'  \ \bigg) \, dr .
\label{term1}\end{equation}
The $s$ integral just gives a constant. In the second integral, we change variable to $r' = r-s'$, and note that $r' \geq r-r/2 \geq R/2$. The $s'$ integral again gives a constant, and we get an upper bound of the form
\begin{equation}
2C \int_{r' \geq R/2}  |\hat F(r')|^2  \, dr' .
\end{equation}
We can insert a factor $(2R)^{-(2m+1)} \ang{r'}^{n+1}$, since $r' \geq 2R$ and $n+1 \geq 2m+1$. This finally gives an estimate of the form $C \| F \|_{H^{(n+1)/2}}^2$ for the first term of \eqref{twoterms}.

For the second term of \eqref{twoterms}, we estimate $\ang{s}^{-m-1} \leq  C \ang{r}^{-m-1}$. This allows us to estimate this term by
 \begin{equation}
2 \| \hat F \|_{L^1}^2 \int_{r \geq R} \ang{r}^{-2m-2}  \, dr \leq C \| F \|_{H^{1/2 + \epsilon}}^2 R^{-(2m+1)} \text{ for any } \epsilon > 0.
\end{equation}
This completes the proof.
\end{proof}

We return to \eqref{bj2}, which we write in the form
\begin{equation}
\alpha^{-n+2j} \int_{\{r > 1\}} \bigg|\int_{0}^\infty F(\alpha
\blambda) e^{ i \blambda r} (\alpha \blambda)^{n/2-j}   \, d\blambda
\bigg|^2  \, dr. \label{bj3}\end{equation} We change variables to
$\blambda' = \alpha \blambda$ and $r' = r/\alpha$. We also choose
$\phi \in C_c^\infty(\RR)$ to be identically $1$ on the support of
$F$, and write $G(\blambda') = \theta(\blambda') {\blambda'}^{n/2-j}
\phi(\blambda')$. The integral becomes
\begin{equation}
\alpha^{-n-1+2j} \int_{\{r' > 1/\alpha \}}
\bigg|\int_{-\infty}^\infty F(\blambda') G(\blambda') e^{ i
\blambda' r'}   \, d\blambda' \bigg|^2  \, dr'.
\label{bj4}\end{equation} The $\blambda'$ integral gives us $(\hat F
* \hat G)(r')$. Applying Lemma~\ref{lem:FG} with $m=n/2-j$ and $R = \alpha^{-1}$, we see that \eqref{bj4} is
bounded uniformly in $\alpha$, as required.

We next consider the terms involving $c$, $a_\pm$, and $\tilde a_{\pm}$. The argument for all these terms is similar, so just consider $c$. In this case, we need a uniform bound on
\begin{equation}
\int_{r > 1} \Big| \int_0^\infty F(\alpha \blambda) e^{\pm i\blambda
r} c(\blambda, z', r, \omega) \, d\blambda \Big|^2 \, dr.
\label{cterm}\end{equation} We use the identity
$$
e^{\pm i\blambda r} = \pm \frac1{ir} \frac{d}{d\blambda} e^{\pm
i\blambda r}.
$$
We integrate by parts. This gives us
\begin{equation}
\int_{r > 1} \Big| \int_0^\infty e^{\pm i\blambda r}
\frac{d}{d\blambda} \Big( F(\alpha \blambda)  c(\blambda, z', r,
\omega) \Big) \, d\blambda \Big|^2 \, \frac1{r^2} dr.
\label{c}\end{equation} When the derivative falls on $F$, we get
$\alpha F'(\alpha \blambda)$. Since $F \in H^1(\RR)$, with compact
support, the function $\alpha F'(\alpha \blambda)$ is $L^1$, with
$L^1$ norm uniformly bounded in $\blambda$. Since $c$ is uniformly
bounded, this gives us a uniform bound on the $\blambda$ integral in
\eqref{c}. When the derivative falls on $c$, using the symbol
estimates, we find that $d_\blambda c$ is integrable in $\blambda$,
and then we can use the fact that $F \in L^\infty(\RR)$ to see that
in this case also, the $\blambda$ integral in \eqref{c} is uniformly
bounded. Finally, the $r$ integral is convergent, so that
establishes the uniform bound on \eqref{cterm}.

Using the inequalities
\begin{equation}
xx' \leq C \rho_L \rho_R  \leq C' e^{-r}, 
\end{equation}
(where the second inequality follows from Proposition~\ref{prop:dist}), the same argument works for the $a_\pm$ and $\tilde a_{\pm}$ terms.

This completes the proof of Theorem~\ref{thm:sm}.

\begin{flushleft}
\vspace{0.3cm}\textsc{Xi Chen\\Department of Mathematics\\Mathematical Sciences
Institute\\Australian National University\\Canberra 0200\\Australia}

\emph{Current address}: \\
\textsc{Riemann Center for Geometry and Physics\\Leibniz Universit\"{a}t Hannover\\Hannover 30167\\Germany}

\emph{E-mail address}: \textsf{xi.chen@anu.edu.au}

\end{flushleft}

\begin{flushleft}
\vspace{0.3cm}\textsc{Andrew Hassell\\Department of Mathematics\\Mathematical Sciences
Institute\\Australian National University\\Canberra 0200\\Australia}

\emph{E-mail address}: \textsf{andrew.hassell@anu.edu.au}

\end{flushleft}

\end{document}